\RequirePackage[l2tabu, orthodox]{nag} 

\documentclass[10pt, a4paper]{article}

\pdfoutput=1 

\usepackage[T1]{fontenc}
\usepackage[utf8]{inputenc}

\usepackage{scrextend} 
\usepackage[nochapters]{style}

\usepackage[runin]{abstract} 

\abslabeldelim{. ---}
\usepackage[oldstyle,    
            intlimits,   
            frenchstyle, 
            widermath,   
            narrowiints, 
            partialup,   
            nowarning    
           ]{kpfonts}

\include{glyphtounicode} 
\usepackage{eucal}    
\usepackage{manfnt}   
\usepackage{enumitem} 
\setlist[itemize]{label=$\triangleright$}

\usepackage{tikz-cd}
\usetikzlibrary{calc, babel} 

\usepackage{xfrac}  

\usepackage{lipsum}                                        
\usepackage[french, colorinlistoftodos, shadow]{todonotes} 
\reversemarginpar

\usepackage[british]{babel}

\ifpdf\PassOptionsToPackage{pdftex,         
                      hyperfootnotes=false, 
                      pdfpagelabels}        
                     {hyperref}             
\hypersetup{
            frenchlinks,                    
            colorlinks=true,                
            linktocpage=true,               
            pdfstartview={FitV},            
            breaklinks=false,               
            pdfpagemode=UseNone,            
            pageanchor=true,                
            plainpages=false,               
            bookmarksnumbered,              
            bookmarksopen=true,             
            bookmarksopenlevel=1,           
            hypertexnames=true,             
            linkcolor=RoyalBlue,            
            citecolor=Maroon,               
            urlcolor=RoyalBlue,             
            pdfhighlight=/O,                
            pdftitle={Homotopy Theory of coalgebras},
            pdfauthor={\textcopyright\ Brice Le Grignou and Damien Lejay},
            pdfsubject={},
            pdfkeywords={},
            pdfcreator={pdfLaTeX},
            pdfproducer={LaTeX with hyperref and classicthesis}
}
\fi

\usepackage{amsthm}               
\usepackage[nameinlink]{cleveref} 

\usepackage{bbding}   
\usepackage{etoolbox} 

\newtheoremstyle{Bourbaki} 
	{\topsep}  
	{\topsep}  
	{\itshape} 
	{ }        
	{\scshape} 
	{. ---}    
	{ }        
	{}         

\newtheoremstyle{Bourbaki-remarque}
	{\topsep}
	{\topsep}
	{}
	{ }
	{\itshape}
	{. ---}
	{ }
	{}
 
\newtheoremstyle{Bourbaki-manicule}
	{\topsep}
	{\topsep}
	{\itshape}
	{ }
	{\scshape}
	{. ---}
	{ }
	{{\normalfont \llap{\HandCuffRight~~}} 
	  \thmname{#1}\thmnumber{ #2}\thmnote{ \normalfont(#3)}}

\newtheoremstyle{Bourbaki-petit}
	{3pt}
	{3pt}
	{\small}
	{ }
	{\itshape}
	{. ---}
	{ }
	{}

\theoremstyle{Bourbaki}
	\newtheorem{proposition}{Proposition}[section]
	\newtheorem{lemma}[proposition]{Lemma}
	\newtheorem{corollary}[proposition]{Corollary}
	\newtheorem{definition}[proposition]{Definition}
	\newtheorem*{definition*}{Definition}
	\newtheorem*{proposition*}{Proposition}
	\newtheorem*{corollary*}{Corollary}
	\newtheorem{theorem}[proposition]{Theorem}
	\newtheorem*{theorem*}{Theorem}
	\newtheorem*{conjecture*}{Conjecture}
 
\theoremstyle{Bourbaki-remarque}
	\newtheorem{remark}[proposition]{Remark}
	
	\newtheorem*{remark*}{Remark}
	\newtheorem*{remaks*}{Remarks}

\theoremstyle{Bourbaki-petit}
	\newtheorem*{smallremark}{Remark}
	\appto\smallremark{\leftskip\parindent}
	\newtheorem*{smallexample}{Example}
	\appto\smallexample{\leftskip\parindent}

\patchcmd{\thmhead}{(#3)}{#3}{}{} 

\makeatletter
\renewenvironment{proof}[1][\proofname]{\par
\pushQED{\qed}%
\normalfont \topsep6\p@\@plus6\p@\relax
\trivlist
\item\relax
{\itshape
#1\@addpunct{. ---}}\hspace\labelsep\ignorespaces
}{%
\popQED\endtrivlist\@endpefalse
}
\makeatother

\usepackage[all, warning]{onlyamsmath}

\usepackage[shortcuts]{extdash} 

\usepackage{prodint} 
\usetikzlibrary{arrows,calc,shapes,decorations.pathreplacing,
decorations.pathmorphing}

\begin{document}


\title{\color{Maroon}\rmfamily\normalfont\spacedallcaps{Exponentiable Higher Toposes}}
\author{\spacedlowsmallcaps{Mathieu Anel} \and \spacedlowsmallcaps{Damien Lejay}}
\date{} 
\maketitle


\begin{abstract}
We characterise the class of exponentiable $\infty$\=/toposes:
$\mathcal X$ is exponentiable if and only if
$\mathcal S\mathrm{h}(\mathcal X)$ is a continuous $\infty$\=/category.
The heart of the proof is the description of the
$\infty$\=/category of $\mathcal C$\=/valued sheaves on $\mathcal X$ as an
$\infty$\=/category of functors that satisfy finite limits conditions as well
as filtered colimits conditions (instead of limits conditions purely); we call
such functors $\omega$-continuous sheaves.

As an application, we show that when $\mathcal X$ is
exponentiable, its $\infty$\=/category of stable sheaves
$\mathcal S\mathrm{h}(\mathcal X, \mathrm{Sp})$ is a dualisable object in
the $\infty$\=/category of presentable stable $\infty$\=/categories.
\end{abstract}



\tableofcontents


\section{Introduction} 

\subsection{Exponentiability of
\texorpdfstring{$\infty$\=/toposes}{infinity-toposes}}
An
$\infty$\=/topos $\mathcal X$ is said to be exponentiable if the functor
$\mathcal Y \mapsto \mathcal X \times \mathcal Y$
has a right adjoint:
$\mathcal Z \mapsto \mathcal Z^\mathcal X$. The idea that exponentiability can
be seen as a form of dualisability is made concrete by the following theorem.

\vspace{0.3cm}

\noindent \HandCuffRight\, {\scshape \bfseries
\Cref{DualisabilityOfStableSheaves}. ---} \emph{Let $\mathcal X$ be an
exponentiable $\infty$\=/topos, then the $\infty$\=/category $\mathcal
S\mathrm{h}(\mathcal X, \mathrm{Sp})$ of stable sheaves on $\mathcal X$ is a
dualisable object of the $\infty$\=/category of presentable stable
$\infty$\=/categories.  }

\vspace{0.3cm}

In order to prove this dualisability result, we characterise the class of
exponentiable $\infty$\=/toposes.
The exponentiability of toposes was studied and understood in the 1981 article
\emph{Continuous
categories and exponentiable toposes} by
\textsc{Johnstone} and \textsc{Joyal}%
~\cite{doi:10.1016/0022-4049(82)90083-4}.
We obtain a characterisation of exponentiable $\infty$\=/toposes by following
a similar proof. Another independent proof of the characterisation of
exponentiable $\infty$\=/toposes has been written by \textsc{Lurie} in
his book \emph{Spectral Algebraic Geometry}%
~\cite[Theorem 21.1.6.12]{lurie2018spectral}.
Our main addition here is the use of the tensor product
of $\infty$\=/categories. In particular in
\cref{thm:le_produit_tensoriel_est_le_coproduit_des_logos} we relate the
tensor product of $\infty$\=/categories with the product in
the $\infty$\=/category of $\infty$\=/toposes.
The characterisation of exponentiable
$\infty$\=/toposes ends up being the same as the one for toposes:

\vspace{0.3cm}
\noindent \HandCuffRight\, {\scshape \bfseries \Cref{ExpTh}. ---} \emph{An
$\infty$\=/topos $\mathcal X$ is exponentiable if and only if the
$\infty$\=/category $\mathcal S\mathrm{h}(\mathcal X)$ is continuous i.e
when the colimit functor $\mathrm{Ind}(\mathcal S\mathrm{h}(\mathcal X))
\to \mathcal S\mathrm{h}(\mathcal X)$ has a left adjoint.  }

\vspace{0.3cm}

The pivot of the exponentiability proof is a rewriting of the
$\infty$\=/category $\mathcal S\mathrm{h}(\mathcal X)$
in terms of finite limits and arbitrary colimits.
When a presentable $\infty$\=/category is continuous it can be obtained as
a coreflective localisation of an $\infty$\=/category of ind-objects:
\[
	\begin{tikzcd}[ampersand replacement=\&]
		\mathrm{Ind}(D) \arrow[r, shift right, swap, "\varepsilon"]
		\&
		\mathcal S\mathrm{h}(\mathcal X)\thinspace,
		\arrow[l, shift right, swap, "\beta"]
	\end{tikzcd}
\]
where $\varepsilon$ is cocontinuous and
$\beta$ is a fully faithful left adjoint to $\varepsilon$. This allows us
to describe the $\infty$\=/category of $\mathcal C$-valued sheaves on
$\mathcal X$ as follows.

\vspace{0.3cm}

\noindent \HandCuffRight\, {\scshape \bfseries
\Cref{thm:faisceaux_dans_un_logos}. ---} \emph{Let $\mathcal X$ be an
exponentiable $\infty$\=/topos, and let $\mathcal C$ be an $\infty$\=/logos.
Then there exists a finitely cocomplete subcategory $D
\subset \mathcal S\mathrm{h}(\mathcal X)$ and a bimodule $w : D^\mathrm{op}
\times D \to \mathcal S$ such that the $\infty$-category of $\mathcal C$-valued
sheaves is equivalent to the $\infty$\=/category of left exact
functors $F : D^\mathrm{op} \to \mathcal C$ satisfying the coend condition:
\[
	F(a) \simeq \int_{b\thinspace \in\thinspace D} w(a,b) \otimes F(b)\thinspace ,
	\quad \text{for all } a \in D\thinspace .
\]
}

Such a description is what we call \emph{$\omega$\=/continuous sheaves}.
In fact, one of the first definitions of sheaves on a topological space $X$
involved Abelian groups associated to \emph{compact subsets} of $X$. A sheaf
was then a functor
$\mathcal F : K \mapsto \mathcal F(K)$
commuting to \emph{finite limits} and specific \emph{filtered colimits}%
~\cite[`faisceaux continus' in the chapter by
\textsc{Houzel}]{doi:10.1007/978-3-662-02661-8}. Namely, a sheaf had to
satisfy the additional condition:
\[
	\mathcal F(K) \simeq \underset{K \ll K'}{\varinjlim}\thinspace
	\mathcal F(K') \thinspace,
\]
where $K \ll K'$ means that there exists an open subset $U$ such that
$K \subset U \subset K'$.
A proof of the equivalence between sheaves on $X$ and
$\omega$\=/continuous sheaves on $X$ is in HTT%
~\cite[Theorem 7.3.4.9]{doi:10.1515/9781400830558}, where they are called
$\mathcal K$-sheaves by
\textsc{Lurie}.

\subsection{Conventions on sizes}

Let $\omega \in \mathbb U \in \mathbb V \in \mathbb W$ be
three Grothendieck universes. To avoid heavy notations,
we establish a dictionary: small means $\mathbb U$-small,
large means $\mathbb V$-small and very large means
$\mathbb W$-small.
By a limit or a colimit, we mean a small one.
By a category or an $\infty$\=/category, we mean a large one.

The large $\infty$\=/category of small spaces will be denoted
$\mathcal S$; its homotopy category is $\mathcal H$. The very large
$\infty$\=/category of large spaces is $\widehat{\mathcal S}$, with
homotopy category $\widehat{\mathcal H}$.
The large $\infty$\=/category of small $\infty$\=/categories is
$\mathcal C\mathrm{at}$; the very large one of
large $\infty$\=/categories will be denoted $\widehat{\mathcal
C\mathrm{at}}$.

\subsection{Acknowledgements}

This work was supported by IBS-$R008$-$D1$.

\section{\texorpdfstring{$\infty$\=/toposes}{Infinity-toposes}}

A standard reference on $\infty$\=/toposes is
HTT~\cite{doi:10.1515/9781400830558}. The reader may also have a look
at \emph{Toposes and homotopy toposes}~\cite{rezk2005toposes} and
\emph{Homotopical algebraic geometry I\@: topos theory}~%
\cite{doi:10.1016/j.aim.2004.05.004}.

\subsection{Definitions}

In this paragraph we recall the definition of an $\infty$\=/topos and
introduce the terminology of $\infty$\=/logoses.

\begin{definition}
We shall say that an $\infty$\=/category $\mathcal L$ is an $\infty$\=/logos
if there exists a small $\infty$\=/category $D$ and an accessible
left exact and
reflective localisation $\mathcal P(D) \to \mathcal L$.

The very large $\infty$\=/category $\mathcal L\mathrm{og}$ of
$\infty$\=/logoses is the non-full subcategory of
$\widehat{\mathcal C\mathrm{at}}$ whose objects are the
$\infty$\=/logoses and the morphisms are the left exact
and cocontinuous functors.
For $\mathcal C$ and $\mathcal D$ two $\infty$\=/categories, the
$\infty$\=/category of left exact and cocontinuous functors shall be denoted
${[\mathcal C, \mathcal D]}^\mathrm{lex}_\mathrm{cc}$.
\end{definition}

\begin{definition}
The very large $\infty$\=/category of $\infty$\=/toposes is
defined by:
\[
	\mathcal T\mathrm{op} = \mathcal L\mathrm{og}^\mathrm{op} \thinspace.
\]
The isomorphism sends an $\infty$\=/topos $\mathcal X$ to its
$\infty$-logos $\mathcal S\mathrm{h}(\mathcal X)$; a morphism
$f : \mathcal X \to \mathcal Y$ is sent to $f^\ast : \mathcal
S\mathrm{h}(\mathcal Y) \to \mathcal S\mathrm{h}(\mathcal X)$.
\end{definition}

\begin{remark}
Manipulating $\infty$\=/topos usually requires many back and forth between
the $\infty$\=/category of $\infty$\=/toposes and its opposite. Distinguishing
by names and notations the two $\infty$\=/categories helps avoiding confusion,
especially between the various types of morphisms.

Distinguishing between a category and its opposite is not new: the
category of affine schemes is the opposite category of the category of rings;
there the equivalence is denoted $A \mapsto \mathrm{Spec}(A)$.
In the same way the category of locales is the opposite category of the
category of frames and the equivalence is denoted $X \mapsto \mathcal O(X)$.

Furthermore, there is a useful analogy between $\infty$\=/logoses and
commutative rings
extending the analogy between colimts and sums; limits and products;
cocomplete categories and Abelian groups.
\end{remark}

\begin{definition}
Let $D$ be a small category. Let $\overline{D}$ be the free
category generated by $D$ by finite limits i.e ${(\overline{D})}^\mathrm{op}$ is
the smallest subcategory in $\mathcal P(D^\mathrm{op})$ containing
$D^\mathrm{op}$ and closed under finite colimits. 
We shall call $\mathcal S[D] = \mathcal P(\overline{D})$
the free $\infty$\=/logos generated by $D$. 
\end{definition}

\begin{proposition}[(Universal property of free $\infty$-logoses)]
Let $D$ be a small $\infty$\=/category and $\mathcal C$ be an
$\infty$\=/logos.  Let $i : D \to \mathcal S[D]$ be the inclusion functor.
Then the restriction functor $i^\ast$
induces an equivalence between the $\infty$\=/category of cocontinuous left
exact functors $\mathcal S[D] \to \mathcal C$ and the $\infty$\=/category of
functors $D \to \mathcal C$.
\end{proposition}

\begin{proof}
Since $\overline{D}$ is the free $\infty$\=/category with finite limits
generated by $D$ and left Kan extensions of left exact $\mathcal C$-valued
functors are still left exact, we have the natural equivalences of
$\infty$\=/categories:
\[
	[D, \mathcal C] \simeq {\left[\overline{D},
	\mathcal C\right]}^{\thinspace\mathrm{lex}}
	\simeq {\left[\mathcal P(\overline{D}),
	\mathcal C\right]}^{\thinspace\mathrm{lex}}_{\thinspace\mathrm{cc}}
	\thinspace ,
\]
induced by the inclusions $D \subset \overline{D}
\subset \mathcal P(\overline{D})$.
\end{proof}

\begin{proposition}
An $\infty$\=/category $\mathcal E$ is an $\infty$\=/logos
if and only if it is a left exact and accessible reflective localisation of a
free $\infty$-logos:
\[
	\begin{tikzcd}
		\mathcal S[D] \arrow[r,"L", shift left] & \mathcal E
		\arrow[l, shift left] \thinspace,
	\end{tikzcd}
\]
that is $L$ is a left exact left adjoint and its right adjoint is fully
faithful and accessible.
\end{proposition}

\begin{proof}
By definition an $\infty$\=/logos $\mathcal E$ is a left exact and accessible
reflective
localisation of a presheaf $\infty$\=/category $L : \mathcal P(D) \to
\mathcal E$ with $D$ a small $\infty$\=/category. The proposition we want to
prove is just a slight variation.  Indeed for any small $\infty$\=/category
$D$, the Yoneda embedding $D \hookrightarrow \mathcal P(D)$ extends to a left
exact and cocontinuous functor $T : \mathcal S[D] \to \mathcal P(D)$. Its right
adjoint is the left extension of the inclusion $D \hookrightarrow \mathcal
P(\overline{D}) = \mathcal S[D]$, it is accessible and fully faithful and $L T
: \mathcal S[D] \to \mathcal E$ is the desired reflective localisation.
\end{proof}

\subsection{Affine \texorpdfstring{$\infty$-toposes}{infinity-toposes}}

The category of commutative rings is generated under colimits by free
rings $\mathbb Z[x_1, \ldots, x_n]$, hence the category affine schemes
is generated under limits by the affine spaces
$\mathbb A^n$. We wish to prove the analogue statement for $\infty$\=/toposes. 

\begin{definition}
An affine $\infty$-topos is an $\infty$\=/topos
$\mathcal X$ such that $\mathcal S\mathrm{h}(\mathcal X)$ is a free
$\infty$-logos. Let $\mathcal A\mathrm{ff}$ be the full
subcategory of $\mathcal T\mathrm{op}$ whose objects are the affine
$\infty$\=/toposes.

We let $\mathbb A^D$ be the affine $\infty$-topos with $\infty$\=/logos
$\mathcal S[D]$, for a small $\infty$\=/category $D$. For convenience,
we also let $\mathbb A$ be the affine $\infty$-topos $\mathbb A^\ast$; its
$\infty$-logos will be denoted $\mathcal S[X]$.
\end{definition}

\begin{proposition}%
\label{AffGen}
The $\infty$\=/category $\mathcal T\mathrm{op}$ is
generated under pullbacks by affine $\infty$-toposes.
\end{proposition}

\begin{proof}
We are going to prove the dual statement that the
$\infty$\=/category $\mathcal L\mathrm{og}$ is generated under pushouts by the
free $\infty$\=/logoses. For any $\infty$-logos
$\mathcal E$, there exists a free $\infty$-logos $\mathcal S[D]$
and a left exact and accessible reflective localisation functor
$L : \mathcal S[D] \to \mathcal S\mathrm{h}(\mathcal X)$.

Let $S$ be the set of morphisms $f$ in $\mathcal S[D]$
such that $L(f)$ is an equivalence in $\mathcal E$, then $S$ is strongly
saturated. Because both $\mathcal
S[D]$ and $\mathcal E$ are accessible $\infty$\=/categories, by proposition
5.5.4.2 in HTT%
~\cite{doi:10.1515/9781400830558}
there exists a small subset
$S_0 \subset S$ such that $S_0$ generates $S$ as a strongly saturated class.

We can now identify $\mathcal E$ as $S_0^{-1} \mathcal S[D]$. Let $J$
be the $\infty$\=/category generated by two objects and one
invertible arrow. We then obtain the following pushout in the
$\infty$\=/category $\mathcal L\mathrm{og}$:
\[
	\begin{tikzcd}
		\mathcal S\left[\underset{S_0}{\coprod}~\Delta^1\right]
		\arrow[r] \arrow[d] \arrow[rd, phantom, "\ulcorner", very near end] 
		& \mathcal S[D] \arrow[d] \\
		\mathcal S\left[\underset{S_0}{\coprod}~J\right]
		\arrow[r] & \mathcal E \thinspace.
	\end{tikzcd}
\]
This ends the proof that any $\infty$-logos is a poushout of free
$\infty$\=/categories of sheaves: morphisms $f^\ast : \mathcal
S\mathrm{h}(\mathcal X) \to \mathcal S\mathrm{h}(\mathcal Y)$ are canonically
equivalent to morphisms $ g^\ast : \mathcal S[D] \to \mathcal
S\mathrm{h}(\mathcal Y)$ such that $g^\ast(s)$ is invertible for any $s \in
S_0$.
\end{proof}

\subsection{Tensor product of
\texorpdfstring{$\infty$\=/categories}{infinity-categories}}

We  gather useful facts from chapter 5.5 of
HTT~\cite{doi:10.1515/9781400830558}, 1.4 and
4.8 of \emph{Higher Algebra}~\cite{lurie2017higher} on tensor products of
$\infty$\=/categories and show that the
coproduct of $\infty$\=/logoses is given by the tensor product of
the underlying cocomplete $\infty$\=/categories.

\begin{theorem}[{\cite[Corollary 4.8.1.4]{lurie2017higher}}]
The very large $\infty$\=/category
$\widehat{\mathcal C\mathrm{at}}_\mathrm{cc}$
of cocomplete $\infty$\=/categories and cocontinuous functors
has a closed symmetric monoidal structure $\otimes$ such that cocontinuous
functors $\mathcal C \otimes \mathcal D \to \mathcal E$ canonically correspond
to functors $\mathcal C \times \mathcal D \to \mathcal E$ cocontinuous in each
variable.

The unit object of $\otimes$ is the cocomplete $\infty$\=/category $\mathcal S$.
\end{theorem}

Given $\mathcal A$, $\mathcal B$, $\mathcal C$ three $\infty$\=/categories,
we shall denote by ${[\mathcal A, \mathcal B]}_\mathrm{cc}$ the
$\infty$\=/category
of cocontinuous functors from $\mathcal A$ to $\mathcal B$. Like wise
${[\mathcal A, \mathcal B]}^\mathrm{c}$ is the $\infty$\=/category of
continuous functors and
${[\mathcal A \times \mathcal B, \mathcal C]}^\mathrm{c,c}$
is the $\infty$\=/category of functors that are continuous in each variable.

It will be important to understand how $\mathcal C \otimes \mathcal D$ is build
as we need these technical details for future proofs. The basic idea
is to force the commutation $c \otimes
\varinjlim d_i \simeq \varinjlim (c \otimes d_i)$ and then to add all the
colimits of `pure tensor' $c \otimes d$. Because $\mathcal C$ and $\mathcal
D$ are large, they are $\mathbb V$-small so we get a reflective localisation
functor:
\[
	\left[{(\mathcal C \times \mathcal D)}^\mathrm{op},
	\widehat{\mathcal S}\thinspace\right] \longrightarrow
	{\left[{(\mathcal C
	\times \mathcal D)}^\mathrm{op},
	\widehat{\mathcal S}\thinspace\right]}^{\thinspace\mathrm{c,c}}\thinspace .
\]
By composition with the Yoneda embedding, we get a functor:
\[
	\mathcal C \times \mathcal D \to
	{\left[{(\mathcal C \times \mathcal D)}^\mathrm{op},
	\widehat{\mathcal S}\thinspace\right]}^{\thinspace\mathrm{c,c}} \thinspace ,
\]
which is cocontinuous in each variable. The tensor product $\mathcal
C \otimes \mathcal D$ is then the smallest cocomplete subcategory of
${\left[{(\mathcal C \times \mathcal D)}^\mathrm{op},
\widehat{\mathcal S}\thinspace\right]}^{\thinspace\mathrm{c,c}}$ that contains
the image of the Yoneda embedding.

\begin{theorem}[{\cite[Remark 4.8.1.18]{lurie2017higher}}]%
\label{thm:pr_est_symetrique_fermee}
Let $\mathcal C$ and $\mathcal D$ be two presentable
$\infty$\=/categories, then $\mathcal C \otimes \mathcal D$ is
presentable.  Moreover ${\left[\mathcal C, \mathcal D\right]}_\mathrm{cc}$ is
also presentable, so that $\mathcal P\mathrm{res}$, the large
$\infty$\=/category of presentable $\infty$\=/categories, inherits a closed
symmetric monoidal structure from
$\widehat{\mathcal C\mathrm{at}}_\mathrm{cc}$.
\end{theorem}

\begin{proposition}[{\cite[Proposition 4.8.1.17]{lurie2017higher}}]%
\label{TensorProductFormula}
Let $\mathcal C$ and $\mathcal
D$ be two presentable $\infty$\=/categories, then
$\mathcal C \otimes \mathcal D
\simeq {\left[\mathcal C^\mathrm{op}, \mathcal D\right]}^\mathrm{c}$.
\end{proposition}

\begin{proposition}[{\cite[Proof of prop. 4.8.1.15]{lurie2017higher}}]%
\label{TensOfLoc}
Let $\mathcal A$ and $\mathcal B$ be presentable
$\infty$\=/categories. Let $\mathcal A \to S^{-1}\mathcal A$ and
$\mathcal B \to T^{-1}\mathcal B$ be accessible and reflective localisations.
Let $f : \mathcal A \times
\mathcal B \to \mathcal A \otimes \mathcal B$ be the canonical map and denote
by $S \boxtimes T$ the set of arrows in $\mathcal A \otimes \mathcal B$ of the
form $f(s \times b)$ with $(s,b) \in S \times \mathcal B$
or $f(a \times t)$ with $(a,t) \in \mathcal A \times T$.
Then the localisation of $\mathcal A \otimes \mathcal B$
along
$S \boxtimes T$ exists, is reflective and accessible. In addition:
\[
{(S \boxtimes T)}^{-1} \mathcal A \otimes \mathcal B \simeq (S^{-1}
\mathcal A) \otimes (T^{-1} \mathcal B) \ .
\]
\end{proposition}

As a direct consequence of the universal property of the tensor product, we
obtain the following corollary:

\begin{corollary}%
\label{PushoutTensor}
The following square is a pushout in
$\widehat{\mathcal C\mathrm{at}}_\mathrm{cc}$,
\[
	\begin{tikzcd}
		\mathcal A \otimes \mathcal B \arrow[r] \arrow[d]
		\arrow[rd, phantom, "\ulcorner", very near end]
		& S^{-1} \mathcal A \otimes \mathcal B \arrow[d] \\
		\mathcal A \otimes T^{-1} \mathcal B \arrow[r] 
		& (S^{-1} \mathcal A) \otimes (T^{-1} \mathcal B)\thinspace.
	\end{tikzcd}
\]
\end{corollary}

For the next theorem, we shall need a lemma that can be found in the
online \emph{corrected version} of HTT.

\begin{lemma}[{\cite[Lemma 6.3.3.4]{htt2017april}}]%
\label{thm:intersection_de_localisations_cc_lex}
Let $\mathcal L$ be an $\infty$\=/logos and let $F, G : \mathcal L \to
\mathcal L$ be two accessible and left exact localisation functors.
Then the intersection
$F\mathcal L \cap G\mathcal L$ is a left exact and accessible localisation
of $\mathcal L$.
\end{lemma}

We now describe the coproducts inside
$\mathcal L\mathrm{og}$. Notice that the following theorem is stated in
HA~\cite[Example 4.8.1.19]{lurie2017higher} but
the proof is left to the reader as it has already been proved in HTT~%
\cite[Theorem 7.3.3.9]{doi:10.1515/9781400830558} for
the case where one of the two $\infty$\=/toposes is localic.

\begin{theorem}%
\label{thm:le_produit_tensoriel_est_le_coproduit_des_logos}
If $\mathcal L$ and $\mathcal M$ are two $\infty$\=/logoses, then
$\mathcal L \otimes \mathcal M$ is
a coproduct of $\mathcal L$ and $\mathcal M$ in  $\mathcal L\mathrm{og}$.
\end{theorem}

\begin{proof}
Let $C$ and $D$ be two small $\infty$\=/categories, we first remark
that

\[
	\mathcal S[C] \otimes \mathcal S[D] \simeq \mathcal S[C \amalg D]\thinspace.
\]

To see this, we remark that

\[
	\mathcal S[C] \otimes \mathcal S[D] = \mathcal
	P(\overline{C}) \otimes \mathcal P(\overline{D})
	\simeq \mathcal P(\overline{C} \times \overline{D})\thinspace.
\]
Now look at the finite completion functor $C
\mapsto \overline{C}$. It starts from $\mathcal C\mathrm{at}$ and goes
to $\mathcal C\mathrm{at}^\mathrm{lex}$,
the large $\infty$\=/category of finitely complete
small $\infty$\=/categories with left exact functors. This functor is left
adjoint to the forgetful functor. Hence it sends coproducts to coproducts. But
in $\mathcal C\mathrm{at}^\mathrm{lex}$ products and coproducts coincide,
and because the forgetful functor preserves limits, we have:
\[
	\overline{C \amalg D} \simeq \overline{C} \times \overline{D}
	\Longrightarrow \mathcal S[C] \otimes \mathcal S[D] \simeq \mathcal P(
	\overline{C} \times \overline{D}) \simeq \mathcal P(\overline{C \amalg D}) =
	\mathcal S[C \amalg D]\thinspace.
\]
By the universal property of free $\infty$\=/logoses, we deduce that
$\mathcal S[C \amalg D]$ a coproduct of $\mathcal S[C]$ and $\mathcal S[D]$.

Let $\mathcal L$ and $\mathcal M$ be two $\infty$\=/logoses, we will now show
that $\mathcal L \otimes \mathcal M$ is an $\infty$\=/logos.
There exists two small
$\infty$\=/categories $C$ and $D$ together with two accessible
left exact reflective
localisation functors $G : \mathcal S[C] \to \mathcal L$
and $H : \mathcal S[D] \to \mathcal M$. Then both
\[
	G^{\overline{D}^\mathrm{op}} :
	{\mathcal S[C]}^{\overline{D}^\mathrm{op}} \to
	{\mathcal L}^{\overline{D}^\mathrm{op}} \quad \text{and} \quad
	H^{\overline{C}^\mathrm{op}} :
	{\mathcal S[D]}^{\overline{C}^\mathrm{op}}
	\to \mathcal M^{\overline{C}^\mathrm{op}}
\]
are left exact and accessible reflective localisation functors.
By \cref{PushoutTensor},
we deduce that $\mathcal L \otimes \mathcal M$ is equivalent to the intersection
$(\mathcal L \otimes \mathcal S[D]) \cap (\mathcal S[C] \otimes \mathcal M)$
and is thus, by \cref{thm:intersection_de_localisations_cc_lex},
an accessible and left exact localisation of $\mathcal S[C]
\otimes \mathcal S[D]$.
As we have
just shown above, $\mathcal S[C] \otimes \mathcal S[D]$ is equivalent to a free
$\infty$-logos, so that $\mathcal L \otimes \mathcal M$ is indeed an
$\infty$-logos.

Let $p^\ast : \mathcal S \to \mathcal L$ be a morphism
of $\infty$\=/logoses (unique up to contractible choice) and let
$q^\ast : \mathcal S \to \mathcal M$ be another. We
claim that the maps $p^\ast \otimes \mathrm{Id}_{\mathcal M} :
\mathcal M \to \mathcal L \otimes \mathcal M$ and $ \mathrm{Id}_{\mathcal L}
\otimes q^\ast : \mathcal L \to \mathcal L \otimes \mathcal M$ exhibit
$\mathcal L \otimes \mathcal M$ as
a pushout of $\mathcal L$ and $\mathcal M$ in $\mathcal L\mathrm{og}$. Notice
that both maps are left exact and cocontinuous: the first is the localisation
along left exact functors of the left exact cocontinuous map $\mathcal S[D]
\to \mathcal S[C] \otimes \mathcal S[D] \simeq \mathcal S[C \amalg D]$ induced
by the canonical map $D \hookrightarrow C \amalg D$. For a symmetric reason,
the second map is also a morphism of $\infty$\=/logoses.

For any $\infty$-logos $\mathcal E$, those two maps induce
a commutative square

\[
	\begin{tikzcd}
		{[\mathcal L \otimes \mathcal M, \mathcal E]}^\mathrm{lex}_\mathrm{cc}
		\arrow[d, hook]
		\arrow[r] & {[\mathcal L, \mathcal E]}^\mathrm{lex}_\mathrm{cc}
		\times {[\mathcal M,\mathcal E]}^\mathrm{lex}_\mathrm{cc}
		\arrow[d, hook] \\
		{[\mathcal S[C \amalg D], \mathcal E]}^\mathrm{lex}_\mathrm{cc} \arrow[r] 
		& {[\mathcal S[C], \mathcal E]}^\mathrm{lex}_\mathrm{cc}
		\times {[\mathcal S[D], \mathcal E]}^\mathrm{lex}_\mathrm{cc} \ .
	\end{tikzcd} 
\]

In the above diagram, the vertical arrows are inclusions and the bottom one is
an equivalence as $\mathcal S[C \amalg D]$ is the coproduct $\mathcal S[C]
\amalg \mathcal S[D]$. 

So we only need to show that if $(\varphi, \psi) \in {[\mathcal S[C], \mathcal
E]}^\mathrm{lex}_\mathrm{cc} \times
{[\mathcal S[D], \mathcal E]}^\mathrm{lex}_\mathrm{cc}$ factorises through
$\mathcal L$ and $\mathcal M$ then the
associated map $\varphi \amalg \psi$ factorises through
$\mathcal L \otimes \mathcal M$. Let $S$ be a
set of arrows of $\mathcal S[C]$ such that $\mathcal L
\simeq S^{-1} \mathcal S[C]$ and let be $T$ such that
$\mathcal M \simeq T^{-1} \mathcal S[D]$. If $\varphi$ and $\psi$
factorise, it means that $\varphi$ sends arrows in $S$ to equivalences and
$\psi$ sends arrows in $T$ to equivalences. Let $S \boxtimes T$ be the set of
arrows of the form $s \otimes x$ for $s \in S, x \in \mathcal S[D]$ or $y
\otimes t$ with $t \in T, y \in \mathcal S[C]$, in $\mathcal S[C] \otimes
\mathcal S[D]$. By the proof that $\mathcal S[C] \otimes \mathcal S[D] \simeq
\mathcal S[C \amalg D]$ above, we have that the map from $\mathcal S[C \amalg
D]$ to $\mathcal E$ associated to $(\varphi, \psi)$ is equivalent to the map
$\varphi \otimes \psi : \mathcal S[C] \otimes \mathcal S[D] \to \mathcal E$.
But $\varphi \otimes \psi$ sends arrows in $S \boxtimes T$ to equivalences so
it factorises through $\mathcal L \otimes \mathcal M
\simeq {(S \boxtimes T)}^{-1} \mathcal S[C] \otimes \mathcal S[D]$.
\end{proof}

\section{Coends for
\texorpdfstring{$\infty$\=/categories}{infinity-categories}}%
\label{coends}
As a prerequisite for the study of dualisable objects in $\widehat{\mathcal
C\mathrm{at}}_\mathrm{cc}$
and the $\infty$\=/category of $\omega$\=/continuous sheaves, we must
develop the theory of coends in the $\infty$\=/setting. Traditional references
on ends and coends for categories include \textsc{MacLane}%
~\cite{doi:10.1007/978-1-4757-4721-8}
and \textsc{Kelly}%
~\cite{isbn:0-521-28702-2}.
An introduction is given by \textsc{Upmeier}%
~\cite{upmeierendcoend}. The beginning of the theory of coends for
quasi-categories has been developed by \textsc{Cranch}%
~\cite{arXiv:1011.3243}
and \textsc{Glasman}%
~\cite{doi:10.1007/s00029-016-0228-z}.
In his thesis
\textsc{Cranch} develops the definition of dinatural transformations between
bifunctors and proves it extends the usual definition for categories,
while \textsc{Glasman} proves
that the space of  natural transformations can be written as an end. 

The tensor product of $\infty$\=/category of presheaves allow us to
develop the theory of coends in a straightforward way. It also does not depend
on a particular model for $\infty$\=/categories.

In this section we let $D$ be a small $\infty$\=/category and $\mathcal C$ be
a cocomplete $\infty$\=/category.

\subsection{Definition and first properties}

\begin{definition}
The left Kan extension of the map functor 
$\mathrm{Map}_D : D^\mathrm{op} \times D \to \mathcal S$
to $\mathcal P(D^\mathrm{op} \times D)$ is called the coend functor
and is denoted:
\[
	\int_D : \mathcal P(D^\mathrm{op} \times D) \to \mathcal S\thinspace.
\]
\end{definition}

\begin{remark}
By definition $\int_D$ is cocontinuous.
\end{remark}

\begin{proposition}%
\label{CoendCCCC}
The functor $\mathcal P(D^\mathrm{op}) \times \mathcal P(D) \to \mathcal S$
defined by:
\[
	(F,G) \mapsto \int_{c \thinspace\in\thinspace C} F(c) \times G(c)\thinspace,
\]
is cocontinuous in each variable.
\end{proposition}

\begin{proof}
This functor is the composition of the cocontinuous functor
$\int_D$ with the canonical map
$\mathcal P(D^\mathrm{op}) \times \mathcal P(D) \to
\mathcal P(D^\mathrm{op})\otimes \mathcal P(D)
\simeq \mathcal P(D^\mathrm{op} \times D)$
which is cocontinuous in each variable.
\end{proof}

Thanks to the tensor product it is possible to extend the definition of the
coend to bimodules with values in any cocomplete $\infty$\=/category $\mathcal
C$.

\begin{proposition}%
\label{PropdefCoend}
The coend functor induces a cocontinuous functor
$\int_D : [D \times D^\mathrm{op}, \mathcal C] \longrightarrow \mathcal C$
still called the coend functor.
\end{proposition}

\begin{proof}
The map is obtained by tensoring with $\mathrm{Id}_\mathcal C$.
We then have a cocontinuous functor
$\int_D \otimes \, \mathrm{Id}_\mathcal C : \mathcal P(D^\mathrm{op} \times D)
\otimes \mathcal C \to \mathcal C$.
But the tensor product $\mathcal P(D^\mathrm{op} \times D) \otimes \mathcal C$
is canonically equivalent to the $\infty$\=/category
$[D \times D^\mathrm{op}, \mathcal C]$.
\end{proof}

\begin{proposition}[(Fubini)]
Let $C$ and $D$ be two small $\infty$\=/categories
and $\mathcal E$ be a cocomplete $\infty$\=/category. For any functor
$F : C^\mathrm{op} \times C \times D^\mathrm{op} \times D \to \mathcal E$
we have:
\[
	\int_{c \thinspace\in\thinspace C}\left[~
	\int_{d \thinspace\in\thinspace D}
	F(-,-,d,d)\right](c,c) \simeq
	\int_{d \thinspace\in\thinspace D} \left[~
	\int_{c \thinspace\in\thinspace C}
	F\left(c,c,-,-\right)\right](d,d) \ .
\]
\end{proposition}

\begin{proof}
Thanks to the equivalences
\[
	[C^\mathrm{op} \times C \times D^\mathrm{op} \times D, \mathcal E] \simeq
	\mathcal P(C \times C^\mathrm{op}) \otimes [D^\mathrm{op} \times
	D, \mathcal E] \simeq \mathcal P(D \times D^\mathrm{op})
	\otimes [C^\mathrm{op} \times C, \mathcal E]\thinspace ,
\]
the following coend square commutes:
\[
	\begin{tikzcd}
		{[C^\mathrm{op} \times C \times D^\mathrm{op} \times D, \mathcal E]}
		\arrow[dd, "\int_C"]
		\arrow[r, "\int_D"] & \mathcal P(C \times C^\mathrm{op}) \otimes \mathcal E
		\arrow[dd, "\int_\mathrm{c}"]
		\\ \\
		\mathcal P(D \times D^\mathrm{op}) \otimes \mathcal E \arrow[r, "\int_D"] &
		\mathcal E\thinspace.
	\end{tikzcd}
\]
\end{proof}

\subsection{Yoneda lemma}
Let us recall the definition of tensoring of a cocomplete
$\infty$\=/category over $\mathcal S$.
Let $(K,c)
\in \mathcal S \times \mathcal C$. Then the cotensor $c^K$ defined by:
\[
	K \otimes c = \underset{K}{\varinjlim}\thinspace c \ .
\]
It is a cocontinuous functor in each variable.

\begin{proposition}[(Yoneda)]%
\label{coYoneda}
Let $F : D \to \mathcal C$ be a functor.
Then for any
$c \thinspace\in\thinspace D$,
\[
	F(c) \simeq \int_{d \thinspace\in\thinspace C} [d,c] \otimes F(d)\thinspace ,
\]
where $[d,c]$ is a shorthand notation for the $\mathrm{Map}(d,c)$.
\end{proposition}

\begin{proof}
Let's prove the case where
$\mathcal C = \mathcal S$. Let $F : D \to \mathcal S$ be any functor and let $y
: D^\mathrm{op} \to \mathcal P(D^\mathrm{op})$ be the Yoneda embedding.

By cocontinuity of the coend functor, we have for $c \in D$,
\[
	\int_{d \thinspace\in\thinspace D} [d,c] \times F(d) \simeq
	\underset{x \thinspace\in\thinspace \mathrm{el}(F)}{\varinjlim}
	\int_{d \thinspace\in\thinspace D} [d,c] \times [x,d]\thinspace.
\]
But by definition of the coend functor
$\int_{d \thinspace\in\thinspace D} [d,c] \times [x,d] \simeq x(c)$
And the formula is proved:
\[
	\int_{d \thinspace\in\thinspace D} [d,c] \times F(d) \simeq
	\underset{x \thinspace\in\thinspace \mathrm{el}(F)}{\varinjlim}
	[x,c] \simeq F(c)\thinspace.
\]

Hence the functor $F \mapsto \int_{d \thinspace\in\thinspace D} [d,-]
\times F(d)$ is
homotopic to the identity. Tensoring it with the identity of $\mathcal C$, we
obtain an endofunctor of $[D, \mathcal C]$
$F \mapsto \int_{d \thinspace\in\thinspace D} [d,-] \otimes F(d)$
homotopic to the identity, which proves the formula.
\end{proof}

\subsection{Left Kan extensions as coends}

When a bimodule $D^\mathrm{op} \times D \to \mathcal C$ is given by the tensor
product of two functors, the coend is easily expressible in terms of colimits.
In return, we are able to calculate left Kan extension along the Yoneda
embedding with the coend functor.

\begin{proposition}%
\label{CoendAsColim}
Let $G$ be an object of $\mathcal P(D)$
and let $F : D \to \mathcal C$ be a functor. Then:
\[
	\int_{d\thinspace \in\thinspace D} G(d) \otimes F(d) \simeq
	\underset{d\thinspace \in\thinspace \mathrm{el}(G)}{\varinjlim}\thinspace
	F (d)\thinspace.
\]
\end{proposition}

\begin{proof}
The functor $\mathcal P(D) \times [D,\mathcal C] \to [D^\mathrm{op}
\times D, \mathcal C]$ sending $(G,F)$ to $G \otimes F$ is cocontinuous in the
first variable and the coend functor is cocontinuous. We then have that
$G \simeq \underset{c\thinspace \in
\thinspace\mathrm{el}(G)}{\varinjlim}\thinspace [-,c]$ implies:
\[
	\begin{tikzcd}[row sep=0pt, column sep=0pt]
		\displaystyle\int_{d\thinspace \in\thinspace D}
		G(d) \otimes F(d)
		& \simeq
		& \displaystyle\int_{d\thinspace \in\thinspace D}\Bigg(
		\underset{d\thinspace \in\thinspace
		\mathrm{el}(G)}{\varinjlim}\thinspace
		[c,d]\Bigg) \otimes F(c)
		&
		\\
		\simeq \displaystyle\int_{d
		\in D}\underset{d\thinspace \in\thinspace
		\mathrm{el}(G)}{\varinjlim}\thinspace
		[d,c] \otimes F(c)
		& \simeq
		&\underset{c\thinspace \in\thinspace \mathrm{el}(G)}{\varinjlim}\thinspace
		\displaystyle\int_{d\thinspace \in\thinspace D} [d,c]
		\otimes F(c)
		&
		\\
		& \simeq & \underset{d\thinspace \in\thinspace
		\mathrm{el}(G)}{\varinjlim}\thinspace
		F(d) \ . & \text{(Yoneda)}
	\end{tikzcd}
\] \end{proof}

\begin{corollary}
Let $F : D \to \mathcal C$ be a functor. Then the
left Kan extension of $F$ along the inclusion $i : D \to \mathcal P(D)$ is
given by:
\[
	\mathrm{Lan}_i F : G \mapsto \int_{d\thinspace \in\thinspace D} G(d)
	\otimes F(d) \thinspace.
\]
\end{corollary}

\begin{proof}
For any functor $G \in \mathcal P(D)$, write $\mathrm{el}(G)$ for
its $\infty$\=/category of elements. Then
$G \simeq \underset{c\thinspace \in\thinspace
\mathrm{el}(G)}{\varinjlim} d$
in $\mathcal P(D)$, so  the left Kan extension is given by:
\[
	(\mathrm{Lan}_i F)(G) \simeq \underset{d\thinspace \in\thinspace
	\mathrm{el}(G)}{\varinjlim}
	\thinspace F(d)\thinspace.
\]
Then apply \cref{CoendAsColim}.
\end{proof}

\begin{remark}
Dualising the proofs, one can obtain the results of the theory of ends.
\end{remark}

\section{Exponentiable \texorpdfstring{$\infty$\=/toposes}{infinity-toposes}}

In this section we prove that exponentiable
$\infty$\=/toposes $\mathcal X$ are those whose $\infty$-logos
$\mathcal S\mathrm{h}(\mathcal X)$ is continuous. This result is an
$\infty$\=/version of the theorem of \textsc{Johnstone} and
\textsc{Joyal}~\cite[Theorem 4.10]{doi:10.1016/0022-4049(82)90083-4}.

\begin{definition}
Let $\mathcal X$ be an $\infty$\=/topos, we will say that
$\mathcal X$ is exponentiable if the functor $\mathcal Y \mapsto
\mathcal Y \times \mathcal X$ has a right adjoint.

For an $\infty$\=/topos $\mathcal Y$ we will say that the particular
exponential $\mathcal Y^\mathcal X$ exists if there is an $\infty$\=/topos
$\mathcal Y^\mathcal X$ and a map $\mathcal X \times \mathcal Y^\mathcal X \to
\mathcal Y$ such that the induced map $\mathrm{Map}(\mathcal Z, \mathcal
Y^\mathcal X) \to \mathrm{Map}(\mathcal Z \times \mathcal X, \mathcal Y)$ is an
isomorphism in $\widehat{\mathcal H}$ for every $\mathcal Z \in \mathcal
T\mathrm{op}$.
\end{definition}

\begin{remark}%
\label{PartExp=>Exp}
By proposition 5.2.2.12 in HTT~\cite{doi:10.1515/9781400830558}, an
$\infty$\=/topos $\mathcal X$ is exponentiable if and only if for any
$\mathcal Y \in \mathcal T\mathrm{op}$,
the particular exponential $\mathcal Y^\mathcal X$ exists.
\end{remark}

\subsection{Injective \texorpdfstring{$\infty$\=/toposes}{infinity-toposes}
            and their points}

\begin{definition}
We shall say that $f : \mathcal X \to \mathcal Y$ is an inclusion or that
$\mathcal X$ is a subtopos of $\mathcal Y$ if $f^\ast$ has a fully faithful
right adjoint.
\end{definition}

\begin{definition}
An $\infty$\=/topos $\mathcal X$ is injective if for every
subtopos $m : \mathcal Y \to \mathcal Z$, the composition morphism
$\mathrm{Map}(\mathcal Z, \mathcal X) \to \mathrm{Map}(\mathcal Y, \mathcal X)$
has a section.
\end{definition}

\begin{remark}
This notion of injective $\infty$\=/topos corresponds to the
notion of weakly injective topos defined in
\emph{Sketches of an Elephant}~\cite{isbn:0-19-851598-7}. We do not
investigate the notions of complete injective and strongly
injective $\infty$\=/toposes.
\end{remark}

\begin{proposition}%
\label{thm:les_injectifs_sont_les_retracts_d_affines}
All affine $\infty$\=/toposes are injective. Furthermore
an $\infty$\=/topos is injective if and only if
it is a retract in $\mathcal T\mathrm{op}$ of an affine $\infty$-toposes.
\end{proposition}

\begin{proof}
Let $\mathcal X$ be an injective $\infty$\=/topos, then by
definition, there exists an inclusion $\mathcal X \to \mathbb A^D$ with $D$ a
small $\infty$\=/category. Because $\mathcal X$ is injective, this morphism must
split.

On the contrary we will prove that any affine $\infty$-topos is injective: let
$\mathcal F = \mathcal S\mathrm{h}(\mathcal Y)$ and $\mathcal G = \mathcal
S\mathrm{h}(\mathcal Z)$ be two $\infty$\=/toposes and $f : \mathcal Y \to
\mathcal Z$ be an inclusion of $\infty$\=/toposes. Thanks to the universal
property of affine $\infty$-toposes, we have the following equivalences
$\mathrm{Map}(\mathcal Y, \mathbb A^D) \simeq \mathrm{core}([D, \mathcal F])$
and $\mathrm{Map}(\mathcal Z, \mathbb A^D) \simeq \mathrm{core}([D, \mathcal
G])$, where $\mathrm{core}$ designates the maximal subgroupoid of an
$\infty$\=/category.
Then the reflective localisation $f^\ast$ gives the desired reflective
localisation ${(f^\ast)}^D$. 

Finally, let's prove that a retract of an injective $\infty$\=/topos is still
injective: let $r : \mathcal X \to \mathcal X'$ be a retraction in $\mathcal
T\mathrm{op}$ with $\mathcal X$ injective and $s : \mathcal X' \to \mathcal X$
a section. Let $i : \mathcal Y \hookrightarrow \mathcal Z$ be an inclusion and
$f : \mathcal Y \to \mathcal X'$ be any map. Then $sf : \mathcal Y \to \mathcal
X$ can be extended in $g : \mathcal Z \to \mathcal X$ because $\mathcal X$ is
injective. Then $rg : \mathcal Z \to \mathcal X$ extends $f$.  \end{proof}

Injective $\infty$\=/toposes have the particular property of being
characterised by their $\infty$\=/categories of points. That is, knowing
$\mathrm{pt}(\mathcal X)$, we can recover $\mathcal X$ in the case where
$\mathcal X$ is injective. 

\begin{remark}
Generally there are several ways to build the opposite category of an 
($\infty$,2)-category
depending on whether one would like to `op' 1-arrows and/or 2-arrows. The
definition of the $\infty$\=/category of points of an $\infty$\=/topos
here reflects the choice of a definition
$\mathcal T\mathrm{op} = \mathcal L\mathrm{og}^\mathrm{op}$
where we choose to `op' only 1-arrows. Having said that, the
definition of the $\infty$\=/category of points of an $\infty$\=/topos
$\mathcal X$ is just
$\mathrm{pt}(\mathcal X) = {[
\thinspace\ast\thinspace, \mathcal X]}_{\mathcal T\mathrm{op}}$
in this ($\infty$,2)-categorical framework.

In the $\infty$\=/category of points, morphisms correspond to `generisation'
of points. Morphisms in the opposite $\infty$\=/category would correspond
to `specialisation' of points.
\end{remark}

\begin{definition}
Let $\mathcal I\mathrm{nj}$ be the full subcategory of $\mathcal T\mathrm{op}$
made of injective $\infty$\=/toposes.

Let us also define $\mathrm{pt}(\mathcal I\mathrm{nj})$ as a  non-full
subcategory of $\widehat{\mathcal C\mathrm{at}}_\infty$.
Its objects are presheaves
$\infty$\=/categories $\mathcal P(D)$ with $D$ a small $\infty$\=/category and
their retracts by $\omega$\=/continuous functors; its morphisms are
the $\omega$\=/continuous functors.
\end{definition}

\begin{proposition}%
\label{PointInj}
The functor of points $\mathrm{pt} : \mathcal
I\mathrm{nj} \to \mathrm{pt}(\mathcal I\mathrm{nj})$ is an equivalence of
$\infty$\=/categories.
\end{proposition}

\begin{proof}
Let $D$ be a small $\infty$\=/category, then $\mathrm{pt}(\mathbb
A^D) \simeq \mathcal P(D^\mathrm{op})$ so, by
\cref{thm:les_injectifs_sont_les_retracts_d_affines},
$\mathrm{pt}$ is a well defined functor from
$\mathcal I\mathrm{nj}$ to $\mathrm{pt}(\mathcal I\mathrm{nj})$. 

We build a new functor $\psi : \mathcal A \mapsto {[\mathcal A, \mathcal
S]}_\omega$, where ${[\mathcal A, \mathcal S]}_\omega$ is the
$\infty$\=/category of $\omega$\=/continuous functors between
$\mathcal A$ and $\mathcal S$.

We claim that $\psi$ is a functor from $\mathrm{pt}(\mathcal I\mathrm{nj})$ to
$\mathcal I\mathrm{nj}^\mathrm{op}$. For this, let $m : \mathcal A \to \mathcal
B$ be an $\omega$\=/continuous functor, then it induces a functor $m^\ast :
{[\mathcal B, \mathcal S]}_\omega \to {[\mathcal A, \mathcal S]}_\omega$.
Because filtered colimits are left exact in $\mathcal S$, we see that finite
limits and all colimits in $\psi(\mathcal A)$ and $\psi(\mathcal B)$ are
computed pointwise, so $m^\ast$ is cocontinous and left exact. Finally,
$\psi(\mathcal P(D^\mathrm{op})) \simeq \mathcal S[D] = \mathcal
S\mathrm{h}(\mathbb A^D)$, so that by
\cref{thm:les_injectifs_sont_les_retracts_d_affines} $\psi$ is well defined. 

By the above computation, the functor of points $\mathrm{pt}$ induces an
equivalence on the subcategory of affine $\infty$-toposes to the subcategory
of $\mathrm{pt}(\mathcal I\mathrm{nj})$ made of presheaves
$\infty$\=/categories.  This equivalence extends to their Cauchy-completion.
\end{proof}

\subsection{Continuous
\texorpdfstring{$\infty$\=/categories}{infinity-categories}}

The definition of a continuous category was first
given in \emph{Continuous categories and exponentiable toposes}~%
\cite{doi:10.1016/0022-4049(82)90083-4}. We shall prove here the
same propositions in the $\infty$\=/setting.

\begin{definition}
Let $\mathcal C$ be an $\infty$\=/category with filtered
colimits. We will say that $\mathcal C$ is continuous if the
colimit functor $\varepsilon : \mathrm{Ind}(\mathcal C) \to \mathcal C$ has
a left adjoint $\beta : \mathcal C \to \mathrm{Ind}(\mathcal C)$.
\end{definition}

The $\infty$\=/category $\mathrm{Ind}(\mathcal C)$ is
not presentable even when $\mathcal C$ is.
We will now focus on continuous categories with smallness
properties. Namely, we wish to replace $\mathrm{Ind}(\mathcal C)$ by a
presentable $\infty$\=/category $\mathrm{Ind}(D)$.

\begin{definition}
Let $\mathcal C$ be a continuous $\infty$\=/category. If
there exists a small full subcategory $D \subset \mathcal C$, such that $D$ is
stable in $\mathcal C$ under finite limits and colimits and such that the
evaluation functor $\mathrm{Ind}(D) \to \mathcal C$ has a fully faithful left
adjoint, then we call the triple adjunction: 
\[
\begin{tikzcd}
\mathrm{Ind}(D) \arrow[r, "\varepsilon" description] 
& \mathcal C \arrow[l, "\alpha", shift left=2]
\arrow[l, "\beta", shift right=2,swap] \thinspace,
\end{tikzcd}
\]
a standard presentation.
\end{definition}

\begin{proposition}%
\label{thm:presentation_standard}
Let $\mathcal C$ be a presentable and continuous $\infty$\=/category.
Then $\mathcal C$ has a standard presentation.
\end{proposition}

\begin{proof}
Because $\mathcal C$ is presentable, there exists a small
and dense full subcategory $D \subset \mathcal C$. We can then take $D'$ the
smallest full subcategory of $\mathcal C$ containing $D$ and closed in
$\mathcal C$ under finite limits and colimits. As such, $D'$ is dense in
$\mathcal C$ so that the evaluation functor $\varepsilon : \mathrm{Ind}(D') \to
\mathcal C$ has a fully faithful right adjoint $\alpha : \mathcal C \to
\mathrm{Ind}(D')$.

As $\varepsilon : \mathrm{Ind}(\mathcal C) \to \mathcal C$ is continuous and
$\mathrm{Ind}(D') \subset \mathrm{Ind}(\mathcal C)$ commutes with limits, we
deduce that $\varepsilon : \mathrm{Ind}(D') \to \mathcal C$ is continuous and
then has a left adjoint $\beta$ because $\mathrm{Ind}(D')$ and $\mathcal C$ are
presentable.
\end{proof}

\begin{proposition}\label{Ind=Cont}
Let $\mathcal D$ be an $\infty$\=/category.
Then $\mathrm{Ind}(\mathcal D)$ is continuous.
\end{proposition}

\begin{proof}
We denote a generic object of
$\mathrm{Ind}(\mathrm{Ind}(\mathcal D))$ as $``\varinjlim_I''
`\varinjlim_{J_i}' d_{ij}$.

Then, the functor $\alpha : \mathrm{Ind}(\mathcal D) \to
\mathrm{Ind}(\mathrm{Ind}(\mathcal D))$ right adjoint to $\varepsilon$ is given
by $``\varinjlim'' d_i \mapsto `\varinjlim' d_i$.

We claim that the left adjoint $\beta$ is given by sending $``\varinjlim'' d_i$
in $\mathrm{Ind}(\mathcal D)$ on $``\varinjlim'' d_i$ in
$\mathrm{Ind}(\mathrm{Ind}(\mathcal D))$, that is $\beta =
\mathrm{Ind}(\alpha)$.  We have a unit transformation of $(\beta,
\varepsilon)$: $\mathrm{Id} \simeq \varepsilon \beta$.  So we can check the
adjunction on mapping spaces.  Any $d \in \mathcal D$ is an
$\omega$-compact object of $\mathrm{Ind}(\mathcal D)$, so that for any $d =
``\varinjlim_i d_i''$, $\beta(d)$ is a formal colimit of $\omega$-compact
objects of $\mathrm{Ind}(D)$. Let $a = ``\underset{j \in J}{\varinjlim}''
`\underset{k \in K_j}{\varprojlim}' d_{jk}$. Then we have:
\[
	\mathrm{Map}(\beta(d), a) \simeq \underset{j \in J}{\varinjlim}\thinspace
	\underset{k \in K_j}{\varinjlim}\thinspace \underset{i \in I}{\varprojlim}
	\thinspace\mathrm{Map}(d_i, d_{jk}) \simeq \mathrm{Map}(d, \varepsilon(a))
	\thinspace.
\]
\end{proof}

\begin{proposition}%
\label{RetractOfCont}
Any retract by $\omega$\=/continuous
functors of a continuous $\infty$\=/category is continuous.
\end{proposition}

\begin{proof}
Let $r : \mathcal C \to \mathcal D$ be a retraction by
$\omega$\=/continuous functors and suppose $\mathcal C$ is continuous. Let $s$
be an $\omega$\=/continuous section of $r$. Because both commute with filtered
colimits, we have $\varepsilon_\mathcal D \circ \mathrm{Ind}(r) \simeq r \circ
\varepsilon_\mathcal C$ and $s \circ \varepsilon_\mathcal D \simeq
\varepsilon_{\mathcal C} \circ \mathrm{Ind}(s)$.  This means we get the
following retract diagram:
\[
	\begin{tikzcd}
		\mathrm{Ind}(\mathcal D)
		\arrow[r,"\mathrm{Ind}(s)"] \arrow[d,"\varepsilon_\mathcal D"]
		& \mathrm{Ind}(\mathcal C) \arrow[d,"\varepsilon_\mathcal C"]
		\arrow[r, "\mathrm{Ind}(r)"] & \mathrm{Ind}(\mathcal D)
		\arrow[d,"\varepsilon_\mathcal D"] \\
		\mathcal D \arrow[r,"s"] &  \mathcal C \arrow[r, "r"] & \mathcal D
		\thinspace.
	\end{tikzcd}
\]

Let $\theta = \mathrm{Ind}(r) \circ \beta_\mathcal C \circ s$. The functor
$\theta$ is a good candidate to be the left adjoint to $\varepsilon_\mathcal
D$.  Indeed, from the unit transformation $\mathrm{Id} \simeq
\varepsilon_\mathcal C \circ \beta_\mathcal C$ we get $u : \mathrm{Id} \simeq
\varepsilon_\mathcal D \circ \theta$. From the counit transformation
$\beta_\mathcal C \circ \varepsilon_\mathcal C \to \mathrm{Id}$ we also get a
counit transformation $k : \theta \circ \varepsilon_\mathcal D \to
\mathrm{Id}$. Finally $k \theta \circ \theta u: \theta \to \theta$ is homotopic
to the identity transformation.  Unfortunately $\varepsilon_\mathcal D  k \circ
u  \varepsilon_\mathcal D : \varepsilon_\mathcal D \to \varepsilon_\mathcal D$
is \emph{not} homotopic to the identity transformation (in this case, one
would call $\theta$ a weak adjoint). Instead $\varepsilon_\mathcal D k$ is
idempotent.

Fortunately, the category $[\mathcal D, \mathrm{Ind}(\mathcal D)]$ has all
filtered colimits; thus idempotents
split~\cite[corollary 4.4.5.16]{doi:10.1515/9781400830558}. Let $\theta
\overset{\tau}{\longrightarrow} \beta \overset{\sigma}{\longrightarrow} \theta$
be such a splitting. We get a new counit map $k' = k \circ (\sigma
\varepsilon_\mathcal D) : \beta \varepsilon_\mathcal C \to \mathrm{Id}$ and a
new unit map $u' = (\varepsilon_\mathcal D  \tau) \circ u : \mathrm{Id} \simeq
\varepsilon_\mathcal D \beta$. This time $\varepsilon_\mathcal D k' \circ u'
\varepsilon_\mathcal D$ is homotopic to the unit transformation, as well as $k'
\beta \circ \beta u'$.

So $\beta$ is a left adjoint to $\varepsilon_\mathcal D$, hence $\mathcal D$ is
a continuous $\infty$\=/category.
\end{proof}

\begin{proposition}
A presentable $\infty$\=/category $\mathcal C$
is the $\infty$\=/category of points
of an injective $\infty$\=/topos $\mathcal X$ if and only if $\mathcal C$ is
continuous.
\end{proposition}

\begin{proof}
Suppose $\mathcal C \in
\mathrm{pt}(\mathcal I\mathrm{nj})$, then by \cref{PointInj}, we know
that $\mathcal C$ is a retract by $\omega$\=/continuous
functors of an $\infty$\=/category of presheaves $\mathcal P(D)$. But $\mathcal
P(D)$ is finitely presentable, so it is continuous by
\cref{Ind=Cont}, and $\mathcal C$ is continuous by
\cref{RetractOfCont}.

Conversely, assume that $\mathcal C$ is continuous.
Because it is smally
presentable, by \cref{thm:presentation_standard}, we get a standard presentation
$\mathrm{Ind}(D) \to \mathcal C$. In particular,
$\mathcal C$ is a retract by $\omega$\=/continuous
functors of $\mathrm{Ind}(D)$, and $\mathrm{Ind}(D)$ is itself such a retract of
$\mathcal P(D)$, so that $\mathcal C \in
\mathrm{pt}(\mathcal I\mathrm{nj})$.
\end{proof}

\begin{corollary}%
\label{Exp=>Cont}
If $\mathcal X$ is an exponentiable $\infty$\=/topos, then the
$\infty$\=/category $\mathcal S\mathrm{h}(\mathcal X)$ is continuous.
\end{corollary}

\begin{proof}
By \cref{thm:les_injectifs_sont_les_retracts_d_affines} the $\infty$\=/topos
$\mathbb A$ is injective and the functor $(-) \times
\mathcal X$ preserves inclusions, so ${\mathbb A}^{\mathcal X}$ is also
injective.  Now, by definition of $\mathbb A$, we have the equivalence of
$\infty$\=/categories
$\mathrm{pt}(\mathbb A^{\mathcal X})
\simeq  \mathcal S\mathrm{h}(\mathcal X)$
which implies the result.
\end{proof}

\subsection{\texorpdfstring{$\omega$}{omega}-continuous sheaves}

Let $X$ be a locally quasi-compact and Hausdorff topological space and let
$\mathcal C$ be an $\infty$\=/category where filtered colimts are left exact.
Then the $\infty$\=/category of $\mathcal C$-valued sheaves on $X$ has an
alternative description%
~\cite[Corollary 7.3.4.10]{doi:10.1515/9781400830558}:
it is the $\infty$\=/category of functors $\mathcal F : \mathcal K^\mathrm{op}
\to \mathcal C$, where $\mathcal K$ is the poset of compact subsets of $X$,
that preserve finite limits and some filtered colimits. We call 
such sheaves $\omega$\=/continuous sheaves.

We wish to prove that, more generally the category of $\mathcal C$-valued
sheaves on an exponentiable $\infty$\=/topos can be described with small
colimits and finite limits condition instead of small limits conditions.

Let us recall the following definition.
\begin{definition}
By an idempotent comonad on an
$\infty$\=/category $\mathcal C$, we
mean the following data: an endofunctor $W : \mathcal C \to \mathcal C$
together with a natural transformation $\varepsilon : W \Rightarrow
\mathrm{Id}_{\mathcal C}$ such that both $\varepsilon W : W^2 \Rightarrow W$
and $W\varepsilon : W^2 \Rightarrow W$ are point-wise equivalences of
endofunctors. This data is
equivalent to the data of the coreflective subcategory of fixed points of
$W$ inside $\mathcal C$%
~\cite[Proposition 5.2.7.4]{doi:10.1515/9781400830558}.
\end{definition}

\subsubsection{\texorpdfstring{$\omega$}{omega}-continuous sheaves of spaces}

Given an exponentiable $\infty$\=/topos $\mathcal X$, as $\mathcal
S\mathrm{h}(\mathcal X)$ is a continuous $\infty$\=/category we have a standard
presentation:
\[
	\begin{tikzcd}
		\mathrm{Ind}(D) \arrow[r, "\varepsilon" description] &
		\mathcal S\mathrm{h}(\mathcal X)\ . \arrow[l, "\alpha", shift left=2]
		\arrow[l, "\beta", shift right=2,swap]
	\end{tikzcd}
\]
We then obtain an idempotent cocontinuous
comonad $W = \beta \varepsilon$ on $\mathrm{Ind}(D)$ and an identification
between $\mathcal S\mathrm{h}(\mathcal X)$ and the $\infty$\=/category 
$\mathrm{Fix}(W)$ of fixed points of $W$ in $\mathrm{Ind}(D)$.

\begin{definition}
The idempotent comonad $W : \mathrm{Ind}(D) \to
\mathrm{Ind}(D)$ is cocontinuous, we write 
$w : D \nrightarrow D$ 
for the corresponding bimodule. That is for $(a,b) \in D^\mathrm{op}\times
D$, we set
$w(a,b) = \mathrm{Map}_D(a, \beta \varepsilon \, b)$.
\end{definition}

\begin{remark}
An object of $w(a,b)$ is what is called a
\emph{wavy arrow} and denoted
$a \rightsquigarrow b$
in \emph{Continuous categories and
exponentiable toposes}%
~\cite{doi:10.1016/0022-4049(82)90083-4}.
\end{remark}

\begin{proposition}%
\label{ContinuousSheaves}
Let $\mathcal X$
be an exponentiable $\infty$\=/topos together with a standard presentation of
its associated $\infty$\=/logos. Then
$\mathcal S\mathrm{h}(\mathcal X)$ is equivalent to the
$\infty$\=/category of left exact functors
$\mathcal F : D^\mathrm{op} \to \mathcal S$ satisfying the condition:
\[
	\mathcal F(a) \simeq \int_{b\thinspace \in\thinspace D}
	w(a,b) \times \mathcal F(b)\thinspace,
\]
for all $a \in D$.
\end{proposition}

\begin{proof}
Let $i : \mathrm{Ind}(D) \to \mathcal P(D)$ be the canonical
embedding and write $(w \otimes -)$ for the left Kan extension of
$w : D \to \mathcal P(D)$ along $D \to \mathcal P(D)$.  That is for
$\mathcal F : D^\mathrm{op} \to \mathcal S$ we have: 
\[
	w \underset{D}{\otimes} \mathcal F = \int_{b\thinspace \in\thinspace D} w(-,b)
	\times \mathcal F(b)\thinspace.
\]

Now suppose $\mathcal F$ is a left exact functor, then we claim that
$w\otimes
\mathcal F \simeq i W \mathcal F$. Indeed, the comonad $W$ is cocontinuous,
hence it coincides with the left Kan extension of its own restriction to $D$.
Furthermore, the embedding $i$ commutes with filtered colimits and $D$
generates $\mathrm{Ind}(D)$ under filtered colimits, hence $i W$ is also a left
Kan extension of its restriction to $D$.

The next step is to show that the two functors $(w \otimes -)$ and
$iW$ coincide on $D$.
This is true by definition as $w(-,b) = iW b$. The conclusion is that $W$ is a
restriction to $\mathrm{Ind}(D)$ of the functor $(w \otimes -)$. Because of
this, we can deduce that 
$i(\mathrm{Fix}(W)) \simeq \mathrm{Fix}(w \otimes -) \cap i(\mathrm{Ind}(D))$
which proves the theorem: the functor
$\beta i : \mathcal S\mathrm{h}(\mathcal X)
\to \mathrm{Fix}(w \otimes -) \cap i(\mathrm{Ind}(D))$
is an equivalence of $\infty$\=/categories.
\end{proof}

\begin{definition}
Let $(w \otimes -)$ be the left Kan extension of $w : D \to \mathcal
P(D)$ along the Yoneda embedding $y : D \to \mathcal P(D)$.  We call an object
of $\mathrm{Fix}(w_!) \cap i(\mathrm{Ind}(D))$ an
$\omega$\=/continuous sheaf (of
spaces). In other words, an $\omega$\=/continuous sheaf of spaces is a
left exact functor
$D^\mathrm{op} \to \mathcal S$ such that:
\[
	\mathcal F(a) \simeq \int_{b\thinspace\in\thinspace D} w(a,b)
	\times \mathcal F(b)\thinspace,
\]
for all $a \in D$.
\end{definition}

\subsubsection{\texorpdfstring{$\mathcal C$}{C}-valued sheaves}

Let $\mathcal X$ be an $\infty$\=/topos and $\mathcal C$ be any
$\infty$\=/category.  The usual definition of $\mathcal C$\=/valued sheaves on
$\mathcal X$ is the following:
$\mathrm{Sh}(\mathcal X, \mathcal C)
= {\left[{\mathcal S\mathrm{h}(\mathcal X)}^\mathrm{op},
\mathcal C\right]}^{\thinspace\mathrm{c}}$.
However, in the case where $\mathcal C$ is a
bicomplete $\infty$\=/category, we wish to show there is another useful
expression to work with:
$\mathrm{Sh}(\mathcal X, \mathcal C) \simeq \mathcal S\mathrm{h}(\mathcal X)
\otimes \mathcal C$.

This result is a slightly different version of \cref{TensorProductFormula}
where the assumptions on the two $\infty$\=/categories are weakened;
essentially by replacing the presentability condition by a
small generation one.
We begin with the simplest case.

\begin{lemma}
Let $D$ be a small $\infty$\=/category and $\mathcal C$ be a
bicomplete $\infty$\=/category, then
${\left[{\mathcal P(D)}^\mathrm{op}, \mathcal C\right]}^{\thinspace\mathrm{c}}
\simeq \mathcal P(D) \otimes \mathcal C$.
\end{lemma}

\begin{proof}
By \cref{thm:dualisabilite_dans_les_categories_presentables},
$\mathcal P(D)$ is a dualisable
object of $\widehat{\mathcal C\mathrm{at}}_\mathrm{cc}$;
its dual is $\mathcal
P(D^\mathrm{op})$. Because $\mathcal C$ is supposed to be bicomplete, we now
have the equivalences:
\[
	\mathcal P(D) \otimes \mathcal C \simeq {[\mathcal P(D^\mathrm{op}), \mathcal
	C]}_\mathrm{cc} \simeq [D^\mathrm{op}, \mathcal C]
	\simeq {\left[{\mathcal P(D)}^\mathrm{op},
	\mathcal C\right]}^{\thinspace\mathrm{c}} .
\]
\end{proof}

\begin{definition}
A cocomplete $\infty$\=/category $\mathcal C$ shall be called smally generated
if it admits a small and dense subcategory $D \subset \mathcal C$. Equivalently,
$\mathcal C$ is smally generated when it is a reflective localisation of
an $\infty$\=/category of presheaves on a small $\infty$\=/category.
\end{definition}

\begin{proposition}%
\label{UberTensorProductFormula}
Let $\mathcal A$ and $\mathcal
B$ be two $\infty$\=/categories. Suppose that $\mathcal A$ is cocomplete and
smally generated and $\mathcal B$ is bicomplete, then:

\[
	\mathcal A \otimes \mathcal B
	\simeq {[\mathcal A^\mathrm{op}, \mathcal B]}^\mathrm{c} \ .
\]
\end{proposition}

\begin{proof}
Let $D$ be a small $\infty$\=/category and let $S$ be a large set
or arrows of $\mathcal P(D)$
such that $\mathcal A$ is the subcategory of $S$-local objects of $\mathcal
P(D)$.  Let $f : \mathcal P(D) \times B \to \mathcal P(D) \otimes \mathcal B$
be the canonical map and let $T$ be the large set of all morphisms in $\mathcal
P(D) \otimes \mathcal B$ having the form $f(s \times \mathrm{Id}_{b})$ for
every $s \in S$ and $b \in \mathcal B$. Then by \cref{TensOfLoc}, we
have:
\[
	T^{-1} (\mathcal P(D) \otimes \mathcal B) \simeq \mathcal A
	\otimes \mathcal B\thinspace.
\]
By the previous lemma we have $\mathcal P(D) \otimes \mathcal B \simeq
{\left[{\mathcal P(D)}^\mathrm{op}, \mathcal B\right]}^{\thinspace\mathrm{c}}$
so that $\mathcal A \otimes
\mathcal B$ correspond to the $\infty$\=/category of $T'$-local objects of
${\left[{\mathcal P(D)}^\mathrm{op}, \mathcal B\right]}^{\thinspace\mathrm{c}}$,
where $T'$ is the
large set of all morphisms of the form $f'(s \times \mathrm{Id}_b)$ with $f' :
\mathcal P(D) \times \mathcal B \to {[{\mathcal P(D)}^\mathrm{op}, \mathcal
B]}^\mathrm{c}$ the corresponding canonical map.

We only need to check that the $\infty$\=/category
${[\mathcal A^\mathrm{op}, \mathcal B]}^\mathrm{c}$ is the subcategory
of ${[{\mathcal P(D)}^\mathrm{op}, \mathcal B]}^\mathrm{c}$ made of $T'$-local
objects. For this, we draw the following commutative diagram:
\[
	\begin{tikzcd}
		{\left[{\mathcal P(D)}^\mathrm{op},
		\mathcal B\right]}^{\thinspace\mathrm{c}}
		\arrow["\varphi", r] 
		& {\left[{\mathcal P(D)}^\mathrm{op} \times
		\mathcal B^\mathrm{op}, \mathcal S\right]}^{\thinspace\mathrm{c,c}}
		\\
		{[\mathcal A^\mathrm{op}, \mathcal B]}^\mathrm{c}
		\arrow[u] \arrow[r, "\psi"]
		& {[\mathcal A^\mathrm{op} \times \mathcal B^\mathrm{op},
		\mathcal S]}^\mathrm{c,c} \arrow[u] \thinspace,
	\end{tikzcd}
\]
where all arrows are fully faithful.
Let $T''$ be the large set of objects of $\varphi(T')$ and let $F$ be an object
of ${\left[{\mathcal P(D)}^\mathrm{op},
\mathcal B\right]}^{\thinspace\mathrm{c}}$.
The morphism $F$ is $T'$\=/local
if and only if $\varphi(F)$ is $T''$\=/local and $T''$\=/local objects of
${\left[{\mathcal P(D)}^\mathrm{op}
\times {\mathcal B}^\mathrm{op}, \mathcal S\right]}^{\thinspace\mathrm{c,c}}$
are precisely the objects of ${[\mathcal A^\mathrm{op} \times \mathcal
B^\mathrm{op}, \mathcal S]}^\mathrm{c,c}$
by direct computation (use Yoneda lemma and
the proof of proposition 5.5.4.20 of HTT%
~\cite{doi:10.1515/9781400830558}).
Hence, $\varphi(F)$ is $T''$\=/local if and only if it lies in the image of
$\psi$. We have proved the desired equivalence.
\end{proof}

\begin{corollary}%
\label{SheafTensor}
Let $\mathcal X$ be an $\infty$\=/topos and
$\mathcal C$ be a bicomplete $\infty$\=/category, then:
\[
	\mathrm{Sh}(\mathcal X,
	\mathcal C) \simeq \mathcal S\mathrm{h}(\mathcal X) \otimes \mathcal C
	\thinspace.
\]
\end{corollary}

\begin{corollary}%
\label{AdjToTensf}
Let $\mathcal A$ and $\mathcal B$ be two
cocomplete and smally generated $\infty$\=/categories and let $\mathcal C$ be a
bicomplete $\infty$\=/category.  Then for every cocontinuous functor $f :
\mathcal A \to \mathcal B$, the cocontinuous functor $f' = f \otimes
\mathrm{Id}_{\mathcal C}$ has a right adjoint
$f^\ast : {[\mathcal A^\mathrm{op}, \mathcal C]}^\mathrm{c}
\to {[\mathcal B^\mathrm{op},\mathcal C]}^\mathrm{c}$
given by precomposition by $f^\mathrm{op}$.
\end{corollary}

\begin{proof}
For this proof, we need to understand concretely how $f'$ is
built. We draw the diagram:

\[
	\begin{tikzcd}
		\mathcal A \otimes \mathcal C
		\arrow[r]
		\arrow[d, "{f'}", swap, shift right]
		& {\left[\mathcal A^\mathrm{op} \times \mathcal C^\mathrm{op},
		\widehat{\mathcal S}\thinspace\right]}^{\thinspace\mathrm{c,c}}
		\arrow[r, shift right]
		\arrow[d, "Lf_!", swap, shift right] 
		& {\left[\mathcal A^\mathrm{op} \times \mathcal C^\mathrm{op},
		\widehat{\mathcal S}\thinspace\right]}
		\arrow[l, shift right]
		\arrow[d, "f_!", swap,  shift right] \\
		\mathcal B \otimes \mathcal C
		\arrow[u, "f^\ast", swap, shift right]
		\arrow[r] 
		& {\left[\mathcal B^\mathrm{op} \times \mathcal C^\mathrm{op},
		\widehat{\mathcal S}\thinspace\right]}^{\thinspace\mathrm{c,c}}
		\arrow[u, "f^\ast", swap, shift right]
		\arrow[r, shift right]
		& {\left[\mathcal B^\mathrm{op} \times
		\mathcal C^\mathrm{op}, \widehat{\mathcal S}\thinspace\right]} \thinspace.
		\arrow[u, "f^\ast", swap, shift right]
		\arrow[l, shift right]
	\end{tikzcd}
\]
With the horizontal arrows going to the right being fully faithful.

By left Kan extension, we get the functor $f_{!}$;
localising it we have $Lf_{!}$. Then by construction of the tensor product,
$Lf_{!}$ sends the subcategory $\mathcal A \otimes \mathcal C$ to $\mathcal B
\otimes \mathcal C$, the restriction of $Lf_{!}$ to $\mathcal A \otimes \mathcal
B$ is the desired $f'$.

Meanwhile, $f^\ast$ is well defined on the right and restricts to the central
column. The key point is that it can also be restricted to the first column
thanks to \cref{UberTensorProductFormula}.

By proposition 4.3.3.7 in HTT%
~\cite{doi:10.1515/9781400830558},
$f_{!}$ is left adjoint to $f^\ast$. This
implies that $Lf_{!}$ is left adjoint to $f^\ast$ and because the restriction of
an adjunction is still an adjunction, we deduce that $f'$ is left adjoint to
$f^\ast$.
\end{proof}

\begin{remark}
\Cref{SheafTensor,AdjToTensf} imply in
particular that for every topological space $X$, there always exists a
sheafification functor adjoint to the natural inclusion:
\[
	\begin{tikzcd}
		\mathcal P\mathrm{Sh}(X) \otimes \mathcal C
		\arrow[r, shift left] & \mathcal S\mathrm{h}(X)
		\otimes \mathcal C \arrow[l, shift left]\thinspace.
	\end{tikzcd}
\]
where $\mathcal P\mathrm{Sh}(X)$ denotes the $\infty$\=/category of
presheaves in $\mathcal S$ on $X$, as long as $\mathcal C$ is bicomplete.
However this sheafification functor is usually not left exact.
\end{remark}

\begin{corollary}%
\label{thm:tensoriser_avec_cclex_preserve_les_lex}
Let $D$ be a small $\infty$\=/category which has all finite colimits and let
$\varphi : \mathcal C \to \mathcal E$ be a left exact and cocontinuous functor
between bicomplete $\infty$\=/categories. Then the functors
$\mathrm{Id}_{\mathrm{Ind}(D)} \otimes \varphi :
\mathrm{Ind}(D) \otimes \mathcal C \longrightarrow
\mathrm{Ind}(D) \otimes \mathcal E$
and
$(\varphi \circ -) : {[D^\mathrm{op}, \mathcal C]}^\mathrm{lex} \longrightarrow
{[D^\mathrm{op}, \mathcal E]}^\mathrm{lex}$
are canonically equivalent. In particular both are left exact and
cocontinuous.
\end{corollary}

\begin{proof}
We know from
\cref{UberTensorProductFormula}
that the functors $(\mathrm{Ind}(D) \otimes -)$ and
${[D^\mathrm{op}, -]}^\mathrm{lex}$ are equivalent when applied to bicomplete
$\infty$\=/categories.

Let us denode by $L : \mathcal P(D) \longrightarrow \mathrm{Ind}(D)$ the left
adjoint to the natural inclusion of ind-objects inside presheaves. Then by
the previous corollary, the functor $L \otimes \mathcal C : \mathcal P(D)
\otimes \mathcal C \simeq {[D^\mathrm{op}, \mathcal C]} \longrightarrow
\mathrm{Ind}(D) \otimes \mathcal C
\simeq {[D^\mathrm{op}, \mathcal C]}^\mathrm{lex}$ has a right adjoint given
by the inclusion ${[D^\mathrm{op}, \mathcal C]}^\mathrm{lex} \subset
[D^\mathrm{op}, \mathcal C]$. The same is true for $\mathcal E$. By 
functoriality of the tensor product, we know that the following diagram
commutes:
\[
	\begin{tikzcd}[ampersand replacement=\&]
		\mathcal P(D) \otimes \mathcal C
		\arrow[rr, "\mathrm{Id}_{\mathcal P(D)}\thinspace \otimes\thinspace
		\varphi"]
		\arrow[d, "L\thinspace \otimes\thinspace \mathrm{Id}_{\mathcal C}", swap]
		\&\& \mathcal P(D) \otimes \mathcal E
		\arrow[d, "L\thinspace \otimes\thinspace \mathrm{Id}_{\mathcal E}"] \\
		\mathrm{Ind}(D) \otimes \mathcal C
		\arrow[rr, "\mathrm{Id}_{\mathrm{Ind}(D)}\thinspace \otimes
		\thinspace \varphi", swap]
		\&\& \mathrm{Ind}(D) \otimes \mathcal E\thinspace.
	\end{tikzcd}
\]
Under the usual identification, the top map is equivalent to the composition
functor $(\varphi \circ -) : [D^\mathrm{op}, \mathcal C] \to
[D^\mathrm{op}, \mathcal E]$. Since $\varphi$ is assumed to be left exact, we
know it sends left exact functors to left exact functors. From this we get
that the restriction of $(\varphi \circ -)$ to the subcategories of left exact
functors coincide with the tensor product map $\mathrm{Id}_{\mathrm{Ind}(D)}
\otimes \varphi$.
\end{proof}

\subsubsection{\texorpdfstring{$\mathcal C$}{C}-valued
\texorpdfstring{$\omega$}{omega}-continuous sheaves}

Going back to an exponentiable $\infty$\=/topos $\mathcal X$ and a standard
presentation:
\[
\begin{tikzcd}
\mathrm{Ind}(D) \arrow[r, "\varepsilon" description]
& \mathcal S\mathrm{h}(\mathcal X)\thinspace, \arrow[l, "\alpha", shift left=2]
\arrow[l, "\beta", shift right=2,swap]
\end{tikzcd}
\]
we let $\mathcal C$ be a bicomplete $\infty$\=/category. Tensoring the
standard presentation with $\mathcal C$ be get another triple adjunction:

\[
	\begin{tikzcd}
		\mathrm{Ind}(D) \otimes \mathcal C \arrow[r, "{\varepsilon'}"
		description] & \mathcal S\mathrm{h}(\mathcal X) \otimes \mathcal C \ .
		\arrow[l, "{\alpha'}", shift left=2]
		\arrow[l, "{\beta'}", shift right=2,swap]
	\end{tikzcd}
\]

The $\infty$\=/category $\mathrm{Ind}(D) \otimes \mathcal C$ is
canonically equivalent to ${[D^\mathrm{op}, \mathcal C]}^\mathrm{lex}$. In
the same way $\mathcal S\mathrm{h}(\mathcal X) \otimes \mathcal C$ can
be identified
with $\mathcal S\mathrm{h}(\mathcal X, \mathcal C)$. The functors $\beta'$ and
$\varepsilon'$ are given by $\beta' = \beta \otimes \mathrm{Id}_{\mathcal C}$,
$\varepsilon' = \varepsilon \otimes \mathrm{Id}_{\mathcal C}$. And we also
identify $\varepsilon'$ with $\beta^\ast$ and $\alpha'$ with
$\varepsilon^\ast$.

Exactly as in the case of $\omega$\=/continuous sheaves of spaces we
obtain an idempotent cocontinuous comonad $W'$ on
$\mathrm{Ind}(D) \otimes \mathcal C$ by letting
$W = \beta' \varepsilon'$, as well as an identification between the
$\infty$\=/category of $\mathcal C$-valued sheaves on $\mathcal X$ and the
$\infty$\=/category of fixed points of $W'$.

Note $L : \mathcal P(D) \to \mathrm{Ind}(D)$ the localisation functor adjoint
to $i : \mathrm{Ind}(D) \to \mathcal P(D)$, and $L' = L \otimes
\mathrm{Id}_{\mathcal C} : \mathcal P(D) \otimes \mathcal C \to \mathrm{Ind}(D)
\otimes \mathcal C$. Note $i'$ the right adjoint to $L'$.
By construction, in the proof of \cref{ContinuousSheaves}, the left Kan
extension $w_! : \mathcal P(D) \to \mathcal P(D)$ satisfies 
$L w_! \simeq W L$.
and as $w_{!}$ is cocontinuous, also by construction,%
\label{DefOfw'_!}
if $w'_{!}$ is defined as $w_! \otimes \mathrm{Id}_{\mathcal C}$, we get
$L' w'_!  \simeq W' L'$.
This guaranties the following inclusion:
$\mathrm{Fix}(w'_{!}) \cap i'({[D^\mathrm{op}, \mathcal C]}^\mathrm{lex})
\subset i'(\mathrm{Fix}(W'))$.

We define $\mathcal C$-valued $\omega$\=/continuous sheaves as:

\begin{definition}
Let $\mathcal X$ be an exponentiable $\infty$\=/topos together
with a standard presentation with generators $D$ and let $\mathcal C$ be a
cocomplete and finitely complete $\infty$\=/category. We define the
$\infty$\=/category ${\mathcal S\mathrm{h}}_\omega(D, \mathcal C)$ of $\mathcal
C$\=/valued $\omega$\=/continuous sheaves on $\mathcal X$ as the
$\infty$\=/category of left exact functors $\mathcal F : D^\mathrm{op}
\to \mathcal C$ such that:
\[
	\mathcal F(a) =
	\int_b w(a,b) \otimes \mathcal F(b)\thinspace ,
\]
for all $a \in D$. Where
$\otimes$ denotes the canonical tensoring of the cocomplete $\infty$\=/category
$\mathcal C$ over $\mathcal S$.
\end{definition}

We deduce the following proposition.

\begin{proposition}
Let $\mathcal X$ be an exponentiable $\infty$\=/topos, $D$ an
$\infty$\=/category of generators of a standard presentation and $\mathcal C$ a
bicomplete $\infty$\=/category, then $\varepsilon' L' :
{\mathcal S\mathrm{h}}_\omega(D, \mathcal C)
\to \mathcal S\mathrm{h}(\mathcal X, \mathcal C)$ is fully
faithful.
\end{proposition}

The remaining key proposition is to show that $(w \otimes -)$ sends left exact
functors
to left exact functors. For this, we need to make the assumption
that $\mathcal C$ is an $\infty$\=/logos.

\begin{definition}
Let $\mathcal C$ be an $\infty$\=/logos. Let us denote by
$|\thinspace.\thinspace| : \mathcal S \longrightarrow \mathcal C$
the (essentially unique) morphism of $\infty$\=/logoses between $\mathcal S$
and $\mathcal C$.
\end{definition}

\begin{theorem}%
\label{thm:faisceaux_dans_un_logos}
Let $\mathcal X$ be an exponentiable $\infty$\=/topos together with $D$ an
$\infty$\=/category of generators of a standard presentation for
$\mathcal S\mathrm{h}(\mathcal X)$.  Let $\mathcal C$ be an $\infty$\=/logos,
then the embedding
${\mathcal S\mathrm{h}}_\omega(D, \mathcal C)
\longrightarrow \mathcal S\mathrm{h}(\mathcal X,
\mathcal C)$
is an equivalence of $\infty$\=/categories. Besides, this equivalence is
functorial along morphisms of $\infty$\=/logoses.
\end{theorem}

\begin{proof}
From the discussion above, we only need to prove that the coend:
\[
	(w \underset{D}{\otimes} -)
	: \mathcal F \mapsto w \underset{D}{\otimes} \mathcal F =
	\int_{b\thinspace\in\thinspace D} |w(-,b)|\times \mathcal F(b)\thinspace,
\]
sends left exact functors $\mathcal F : D^\mathrm{op} \to \mathcal C$ to
left exact functors. The idea of the proof is exactly the same as the one of
\cref{ContinuousSheaves}, modulo a change of base $\infty$\=/logos to
$\mathcal C$. Using the cocontinuous and left exact map $|\thinspace.\thinspace|
: \mathcal S \to \mathcal C$, the $\infty$\=/category $D$ becomes enriched
over $\mathcal C$. Given an object $d \in D$, the functor
$|\mathrm{Map}_D(-,d)| : D^\mathrm{op} \longrightarrow \mathcal C$
defines an object of ${[D^\mathrm{op}, \mathcal C]}^\mathrm{lex}$ and using the
$\mathcal C$\=/enriched version of the Yoneda embedding, we get a fully
faithful functor:
\[
	D \hookrightarrow {[D^\mathrm{op}, \mathcal C]}^\mathrm{lex}
	\simeq \mathrm{Ind}(D) \otimes \mathcal C\thinspace,
\]
that describes $\mathrm{Ind}(D) \otimes \mathcal C$ as the
$\mathcal C$\=/enriched
$\infty$\=/category freely generated by $D$ under
filtered $\mathcal C$\=/colimits:
$\mathrm{Ind}(D) \otimes \mathcal C \simeq \mathrm{Ind}_{\mathcal C}(D)$.
Moreover the embedding:
\[
	\mathrm{Ind}_{\mathcal C}(D) \simeq {[D^\mathrm{op},\mathcal C]}^\mathrm{lex}
	\hookrightarrow [D^\mathrm{op}, \mathcal C]
	\simeq \mathcal P(D) \otimes \mathcal C\thinspace,
\]
commutes with filtered $\mathcal C$-colimits as they are computed pointwise.

As a consequence, we need only to check that $(w \otimes -)$ sends representable
functors to left exact ones. Let $d$ be an object of $D$, when:
\[
	\int_{b\thinspace\in\thinspace D} |w(-,b)|\times|\mathrm{Map}_D(b,d)|
	\simeq \left|\thinspace\int_{b\thinspace\in\thinspace D} w(-,b)\times
	\mathrm{Map}_D(b,d)\right| \simeq |w(-,d)|\thinspace,
\]
which is left exact as a composite of left exact functors.

The functoriality of the equivalence is a consequence of the functoriality
of the tensor product coupled with
\cref{thm:tensoriser_avec_cclex_preserve_les_lex}.
\end{proof}

\subsection{Exponentiability theorem}

In \cref{Exp=>Cont}, we have seen that the $\infty$\=/category $\mathcal
S\mathrm{h}(\mathcal X)$ is continuous for an exponentiable $\infty$\=/topos
$\mathcal X$. In the next theorem we wish to show the reciprocal statement.

\begin{theorem}%
\label{ExpTh}
An $\infty$\=/topos $\mathcal X$ is exponentiable if
and only if the $\infty$\=/category $\mathcal S\mathrm{h}(\mathcal X)$ is
continuous.
\end{theorem}

The proof of this theorem will follow naturally from the lemmas below.

\begin{lemma}%
\label{ExpIffExpOnAff}
An $\infty$\=/topos $\mathcal X$ is
exponentiable if and only if the particular exponentials
${\left(\mathbb A^D\right)}^\mathcal
X$ exists for every $\mathbb A^D \in \mathcal A\mathrm{ff}$.
\end{lemma}

\begin{proof}
By \cref{PartExp=>Exp}, we only need to show that the
particular exponentials $\mathcal Y^\mathcal X$ exist for every $\mathcal Y \in
\mathcal T\mathrm{op}$. But by \cref{AffGen}, any $\mathcal Y \in
\mathcal T\mathrm{op}$ is a limit of affine $\infty$-toposes i.e $\mathcal Y
\simeq \underset{i \in I}{\varprojlim}\thinspace \mathbb A^{D_i}$.
As every exponential
${\left(\mathbb A^{D_i}\right)}^\mathcal X$ exists, we get a map:
\[
	\mathcal X \times \underset{i \in
I}{\varprojlim}\thinspace {\left(\mathbb A^{D_i}\right)}^\mathcal X \to
\underset{i \in I}{\varprojlim}\thinspace \mathbb A^{D_i}\thinspace,
\]
that exhibits
$\underset{i \in I}{\varprojlim}\thinspace
{\left(\mathbb A^{D_i}\right)}^\mathcal X$ as the
exponential $\mathcal Y^\mathcal X$.
\end{proof}

\begin{lemma}%
\label{ExpOnAff}
Let $\mathcal X$ be an $\infty$\=/topos for
which the exponential $\mathbb A^\mathcal X$ exists, then all exponentials
${\left(\mathbb A^D\right)}^\mathcal X$ exist for every affine $\infty$-topos
$\mathbb A^D$.
\end{lemma}

\begin{proof}
The first part of the proof consists in showing that the
$\infty$\=/topos $\llbracket D \rrbracket$ defined by $\mathcal P(D) = \mathcal
S\mathrm{h}(\llbracket D \rrbracket)$ is exponentiable.

For this, we will show that $\mathcal P(D)$ is coexponentiable in $\mathcal
L\mathrm{og}$. The map $\mathcal S[C] \to \mathcal S[C \times D^\mathrm{op}]$
gives the unit map $\mathcal S[C] \to \mathcal S[C \times D^\mathrm{op}]
\otimes \mathcal P(D)$.  For every $\mathcal L \in
\mathcal L\mathrm{og}$, we then have a map:
\[
\mathrm{Map}_{\mathcal
L\mathrm{og}}(\mathcal S[C \times D^\mathrm{op}], \mathcal L)
\to \mathrm{Map}_{\mathcal L\mathrm{og}}(\mathcal S[C], \mathcal
L \otimes \mathcal P(D))\thinspace,
\]
which is an isomorphism in
$\widehat{\mathcal H}$. Hence by \cref{ExpIffExpOnAff},
$\llbracket D \rrbracket$ is
exponentiable and by the calculation we have just done
${\left(\mathbb A^C\right)}^{\llbracket D \rrbracket}
\simeq \mathbb A^{C\thinspace \times
\thinspace D^\mathrm{op}}$.

We shall end the proof by noticing that the particular
exponential ${\left(\mathbb A^D\right)}^\mathcal X$ can be defined as
${\left({\mathbb A}^\mathcal X\right)}^{\llbracket D^\mathrm{op}\rrbracket}$
for any small $\infty$\=/category $D$.
The evaluation map $\mathcal X \times \mathbb A^\mathcal X \to \mathbb A$ gives:
\[
	\mathcal X^{\llbracket D^\mathrm{op} \rrbracket}
	\times {\left(\mathbb A^\mathcal X\right)}^{\llbracket D^\mathrm{op}
	\rrbracket} \longrightarrow \mathbb A^D\thinspace.
\]
Using the map $\mathcal X \to
\mathcal X^{\llbracket D^\mathrm{op} \rrbracket}$ (from exponential of
the first projection $\mathcal X \times \llbracket D^\mathrm{op} \rrbracket
\to \mathcal
X$), we end up with the evaluation map:
\[
	\mathcal X \times {\left(\mathbb
	A^\mathcal X\right)}^{\llbracket D^\mathrm{op} \rrbracket} \to \mathbb A^D
	\thinspace.
\]
Finally
for every $\infty$\=/topos $\mathcal Y$, we get the following equivalences of
mapping spaces in $\mathcal{T}\mathrm{op}$:
\[
	\left[\mathcal Y, {\left(\mathbb A^\mathcal
	X\right)}^{\llbracket D^\mathrm{op} \rrbracket}\right]
	\simeq \left[\mathcal Y \times \llbracket D^\mathrm{op} \rrbracket
	\times \mathcal X, \mathbb A\right] \simeq
	\left[\mathcal Y \times \mathcal X , \mathbb A^D\right] \thinspace.
\]
Using \cref{ExpIffExpOnAff} again, $\mathcal X$ is
exponentiable.
\end{proof}

\begin{lemma}%
\label{ExpA1}
Let $\mathcal X$ be an $\infty$\=/topos such that
$\mathcal S\mathrm{h}(\mathcal X)$ is a continuous $\infty$\=/category, then the
exponential $\mathbb A^{\mathcal X}$ exists in $\mathcal T\mathrm{op}$.
\end{lemma}

\begin{proof}
Let $\mathcal X$ be an $\infty$\=/topos such that $\mathcal S\mathrm{h}(\mathcal
X)$ is continuous. To show that $\mathbb A^\mathcal X$ exists, we have to find
an injective $\infty$\=/topos $\mathcal I$ and functorial isomorphisms
$\mathrm{Map}_{\mathcal L\mathrm{og}}(\mathcal S\mathrm{h}(\mathcal I),
\mathcal L) \to
\mathrm{Map}_{\mathcal L\mathrm{og}}(\mathcal S[X],
\mathcal L \otimes
\mathcal S\mathrm{h}(\mathcal X))$
in $\widehat{\mathcal H}$.

First we build $\mathcal I$. For this take a standard presentation of $\mathcal
S\mathrm{h}(\mathcal X)$:
\[
	\begin{tikzcd}
		\mathrm{Ind}(D) \arrow[r, "\varepsilon" description] &
		\mathcal S\mathrm{h}(\mathcal X) \arrow[l, "\alpha", shift left=2]
		\arrow[l, "\beta", shift right=2,swap] \thinspace.
	\end{tikzcd}
\]
Let $W = \beta \varepsilon$.
Now because $\beta$ and $\varepsilon$ are adjoint and that $\beta$ is fully
faithful, we have that $W$ is an idempotent cocontinuous comonad on
$\mathrm{Ind}(D)$ and $\beta$ induces an equivalence between $\mathcal
S\mathrm{h}(\mathcal X)$ and the fixed points of $W$.

Let $w : D^\mathrm{op} \times D \to \mathcal S$ be the corresponding bimodule.
Notice that because $W$ has its values in ind-objects, the bimodule $w$ is left
exact in the first variable. Moreover the idempotent comonad structure of $W$
can be rewritten in the following way: the bimodule $w$ bears a bimodule map
$w \Rightarrow \mathrm{Map}_{D}$ inducing the following formula:
\[
	\int_\mathrm{c} w(a,c) \times w(c,b) \simeq w(a,b)\ .
\]

Let us denote by $(- \underset{D^\mathrm{op}}{\otimes} w)
: \mathcal P(D^\mathrm{op})
\to \mathcal P(D^\mathrm{op})$ the functor defined by:
\[
	G \underset{D^\mathrm{op}}{\otimes} w = \int_c G(c) \times w(c,-) \ .
\]
Since $w$ is left exact in the first variable, the endofunctor $(- \otimes w)$
(obtained by left extension from $w$) is cocontinuous and left exact;
it also bears the structure of an idempotent
comonad. We shall call $\mathcal P$ its $\infty$\=/category of fixed points.
We end up with the following presentation:
\[
	\begin{tikzcd}
		\mathcal P(D^\mathrm{op})
		\arrow[r, "\rho" description] &
		\mathcal P
		\arrow[l, "\kappa", shift left=2]
		\arrow[l, "\gamma", shift right=2,swap] \thinspace,
	\end{tikzcd}
\]
where $\gamma \rho \simeq (- \otimes w)$, both $\gamma$ and $\kappa$ are fully
faithful and $\gamma$ is left exact. From this presentation we deduce
immediately
that $\mathcal P$ is an $\infty$\=/logos and as it is a retract in the
$\infty$\=/category of $\infty$\=/logoses of $\mathcal P(D^\mathrm{op})$, its
associated $\infty$\=/topos is injective. This is our $\mathcal I$.

Let $\mathcal L$ be any $\infty$\=/logos. We will show that
${[\mathcal P, \mathcal L]}^\mathrm{lex}_\mathrm{cc}$ and
$\mathcal S\mathrm{h}(\mathcal X) \otimes \mathcal L$
are equivalent by contemplating their respective descriptions.

The $\infty$\=/category ${[\mathcal P, \mathcal L]}^\mathrm{lex}_\mathrm{cc}$
is equivalent,
by definition of $\mathcal P$, to
the $\infty$\=/category of cocontinuous and left exact functors $\mathcal F :
\mathcal P(D^\mathrm{op}) \to \mathcal L$ such that:
\[
	\mathcal F \circ (- \underset{D^\mathrm{op}}{\otimes} w) \simeq \mathcal F \ .
\]
But since $\mathcal F$ is cocontinuous and left exact, this $\infty$\=/category
is also equivalent to the $\infty$\=/category of left exact functors
$\mathcal F : D^\mathrm{op} \to \mathcal L$ such that:
\[
	w \underset{D}{\otimes} \mathcal F \simeq \mathcal F \ .
\]
In other words,  ${[\mathcal P,\mathcal L]}^\mathrm{lex}_\mathrm{cc}$ is
equivalent to ${\mathcal S\mathrm{h}}_\omega(D, \mathcal L)$.
Moreover this equivalence is functorial in $\mathcal L$.

Using
\cref{thm:faisceaux_dans_un_logos}
(one only needs a continuous $\infty$\=/category
to use the conclusions of the theorem), we are also given functorial
equivalences of $\infty$\=/categories between
$\mathcal S\mathrm{h}_\omega(D,\mathcal L)$
and $\mathcal S\mathrm{h}(\mathcal X) \otimes \mathcal L$, so that we obtain
equivalences in $\widehat{\mathcal H}$:
$\mathrm{Map}_{\mathcal{T}\mathrm{op}}(\mathcal Y, \mathcal I)
\simeq \mathrm{Map}_{\mathcal{T}\mathrm{op}}(\mathcal Y \times \mathcal X,
\mathbb A)$
which are functorial in $\mathcal Y$. This proves the existence of
$\mathbb A^{\mathcal X}$.
\end{proof}

\subsection{Glossary of maps between
\texorpdfstring{$\infty$\=/toposes}{infinity-toposes}}

Given a morphism of $\infty$\=/toposes 
$f: \mathcal X \to \mathcal Y$,
we shall say that:
\begin{itemize}
	\item the $\infty$\=/topos $\mathcal X$ has trivial $\mathcal
	      Y$-shape if $f^\ast$ is fully faithful;
	\item the morphism $f$ is essential if $f^\ast$ has a left adjoint;
	\item the morphism $f$ is proper if it satisfies the stable
	      Beck-Chevalley condition%
	      ~\cite[Definition 7.3.1.4]{doi:10.1515/9781400830558};
	\item the morphism $f$ is cell-like if $f$ is proper
	and $\mathcal X$ has trivial $\mathcal Y$-shape;
	\item the morphism $f$ is étale if there exists $U \in \mathcal
	S\mathrm{h}(\mathcal Y)$ such that $f^\ast : \mathcal S\mathrm{h}(\mathcal Y)
	\to {\mathcal S\mathrm{h}(\mathcal Y)}_{/U}$ is the product by $U$;
	\item the $\infty$\=/topos $\mathcal X$ is an open (resp.\ closed, resp.\
	      locally closed) subtopos
	      of $\mathcal Y$ if $f$ is an étale inclusion (resp.\ proper inclusion,
	      resp.\ the intersection of an étale inclusion and a proper inclusion).
\end{itemize}

\subsection{Examples of exponentiable
\texorpdfstring{$\infty$\=/toposes}{infinity-toposes}}

\begin{proposition}%
\label{thm:omega_presentable_implique_exponentiable}
Let $\mathcal X$ be an $\infty$\=/topos
and suppose that the $\infty$\=/category $\mathcal S\mathrm{h}(\mathcal X)$
is $\omega$-presentable.
Then $\mathcal X$ is exponentiable.
\end{proposition}

\begin{proof}
If $\mathcal S\mathrm{h}(\mathcal X)$ is $\omega$-presentable,
then there exists a small $\infty$\=/category $D$ such that
$\mathcal S\mathrm{h}(\mathcal X) \simeq \mathrm{Ind}(D)$
and by \cref{Ind=Cont}, the $\infty$\=/category
$\mathcal S\mathrm{h}(\mathcal X)$ is continuous.
\end{proof}

In particular all the affine $\infty$-toposes $\mathbb A^D$ are
exponentiable. Also all $\infty$\=/toposes $\mathcal X$ such that
$\mathcal S\mathrm{h}(\mathcal X)$ is a presheaf $\infty$\=/category. In
particular if $G$ is a discrete group then $\mathcal B\mathrm{G}$ is an
exponentiable $\infty$\=/topos. 
Another class of examples is given by the locally coherent $n$-toposes.

\begin{definition}
Let $C$ be an $n$-category which admits finite limits. We
will say that a Grothendieck topology on $C$ is finitary if for every
object $c \in C$ and every covering sieve $C^{(0)}_{/c} \subset C_{/c}$ there
exists a finite collection of morphisms ${\{c_i \to c\}}_{i \in I}$ in
$C^{(0)}_{/c}$ which generates the sieve $C^{(0)}_{/c}$.
\end{definition}

\begin{definition}
Let $n < \infty$, an $n$-topos is locally coherent if it is
an $n$-topos associated to a finitary $n$-site.
\end{definition}

\begin{proposition}
Let $n < \infty$ and $\mathcal X$ be a locally coherent
$n$-topos, then $\mathcal X$ is exponentiable.
\end{proposition}

\begin{proof}
If $C$ is a finitary $n$-site, then $\mathcal S\mathrm{h}(C)$ is
$\omega$-presentable. Indeed the sheaf condition for $\mathcal F \in \mathcal
S\mathrm{h}(C)$ boils down to finite limit conditions. All sieves are
generated by finite collections ${\{c_i \to c\}}_{i \in I}$ so the sheaf
condition:
\[
	\begin{tikzcd}
		\displaystyle \prod_i \mathcal F(c_i) \arrow[r, shift left=2]
		\arrow[r, shift right=2]
		& \displaystyle \prod_{i \to j} \mathcal F(c_j) \arrow[l]
		\arrow[r, shift left=4] \arrow[r] \arrow[r, shift right=4] &
		\arrow[l, shift left=2] \arrow[l, shift right=2] \cdots 
	\end{tikzcd}
\]
involves
only finite products at each level and there is only a finite number of levels
because in $\mathcal S_{\leq n-1}$ limits of cosimplicial objects can be
computed after being truncated at level $n$.

The consequence is that the inclusion
$\mathcal S\mathrm{h}(C) \hookrightarrow \mathcal P(C)$
commutes with filtered colimits, which means
that the reflective localisation $\mathcal P(C) \to \mathcal S\mathrm{h}(C)$ is
$\omega$-accessible, so $\mathcal S\mathrm{h}(C)$ is $\omega$-presentable and
$\mathcal X$ is exponentiable by
\cref{thm:omega_presentable_implique_exponentiable}.
\end{proof}

The following two propositions are trivial properties of
exponentiable objects in an $\infty$\=/category with finite limits.

\begin{proposition}
Let $\mathcal X$ and $\mathcal Y$ be two exponentiable
$\infty$\=/toposes, then $\mathcal X \times \mathcal Y$ is exponentiable.
\end{proposition}

\begin{proposition}
Let $\mathcal X$ be an exponentiable $\infty$\=/topos and $r :
\mathcal X \to \mathcal Y$ a retraction. Then $\mathcal Y$ is also
exponentiable.
\end{proposition}

\begin{proposition}
Let $\mathcal X \to \mathcal Y$ be an étale morphism. If
$\mathcal Y$ is exponentiable, so is $\mathcal X$. In particular open
subtoposes of $\mathcal Y$ are exponentiable.
\end{proposition}

\begin{proof}
The $\infty$\=/category ${\mathcal S\mathrm{h}(\mathcal Y)}_{/U}$ is
continuous because colimits in the slice $\infty$\=/topos can be computed using
the projection $\pi_! : \mathcal Y_{/U} \to \mathcal Y$.
\end{proof}

\begin{corollary}%
\label{Qcqs=>Exp}
Let $X$ be a locally quasi-compact and
quasi-separated topological space, then its associated $\infty$\=/topos is an
exponentiable $\infty$\=/topos.
\end{corollary}

\begin{proof}
If $X$ is locally quasi-compact and quasi-separated,
then the frame $\mathcal O(X)$ is a retract of
${\mathrm{Ind}(\mathcal O(X))}_{/X}$. Passing to the associated
$\infty$-toposes and
using \cref{thm:omega_presentable_implique_exponentiable} proves the corollary.
\end{proof}

\begin{remark}
This corollary implies in particular that $\infty$\=/toposes associated to
locally quasi-compact and Hausdorff
topological spaces are exponentiable. An independent proof of
that statement is given in HTT%
~\cite[Theorem 7.3.4.9]{doi:10.1515/9781400830558}.
\end{remark}

The following proposition encompasses some of the previous ones.

\begin{proposition}%
\label{LimDiagExp}
Let $f : I \to \mathcal T\mathrm{op}$ be
a small diagram of exponentiable $\infty$\=/toposes. Suppose also that for any
arrow $i \to j$ in $I$,  the following square commutes:
\[
	\begin{tikzcd}
		\mathrm{Ind}(\mathcal S\mathrm{h}(\mathcal X_j))
		\arrow[rr, "\mathrm{Ind}(f_{ij}^\ast)"]
		&& \mathrm{Ind}(\mathcal S\mathrm{h}(\mathcal X_i)) \\
		\mathcal S\mathrm{h}(\mathcal X_j) \arrow[u, "\beta_j"]
		\arrow[rr, "f_{ij}^\ast"] && \mathcal S\mathrm{h}(\mathcal X_i)
		\arrow[u, "\beta_i"]\thinspace.
	\end{tikzcd}
\]
Then,
\begin{itemize}
	\item the colimit of $f$ is exponentiable;
	\item if $I$ is cofiltered, the limit of $f$ is exponentiable.
\end{itemize}
\end{proposition}

\begin{proof}
By sections 6.3.2 and 6.3.3 in HTT%
~\cite{doi:10.1515/9781400830558},
limits and filtered
colimits of $\infty$\=/categories of sheaves can be computed in
$\widehat{\mathcal C\mathrm{at}}$. By direct computation,
$\varprojlim \mathrm{Ind}(\mathcal S\mathrm{h}(\mathcal X_i)) \simeq
\mathrm{Ind}(\varprojlim \mathcal S\mathrm{h}(\mathcal X_i))$
and thanks to the commuting squares we requested, we get a functor
$\beta : \varprojlim \mathcal S\mathrm{h}(\mathcal X_i)
\to \mathrm{Ind}(\mathcal S\mathrm{h}(\mathcal X_i))$
left adjoint to the evaluation functor, so that
$\varprojlim \mathcal S\mathrm{h}(\mathcal X_i)$ is continuous.

In the same way, if $I$ is cofiltered, then
$\varinjlim \mathrm{Ind}(\mathcal S\mathrm{h}(\mathcal X_i))
\simeq \mathrm{Ind}(\varinjlim \mathcal
S\mathrm{h}(\mathcal X_i))$
and we end up with the same conclusion.
\end{proof}
\begin{remark}
As a consequence, a colimit of a diagram of exponentiable
$\infty$\=/toposes with étale maps is exponentiable.
\end{remark}

\begin{corollary}
Let $f : I \to \mathcal T\mathrm{op}$ be a small cofiltered
diagram of exponentiable
$\infty$\=/toposes. Assume that for every arrow $i \to j$, the corresponding
morphism $f(i) \to f(j)$ is proper and that the $\infty$\=/logos associated
to each $f(i)$ is $\omega$-accessible. Then the limit of $f$ is also
an exponentiable $\infty$\=/topos.
\end{corollary}

\begin{proof}
Let $f : \mathcal X \to \mathcal Y$ be an arrow in such a diagram. Then by
assumption both $\mathcal S\mathrm{h}(\mathcal X)$ and
$\mathcal S\mathrm{h}(\mathcal Y)$ are $\infty$\=/categories of ind-objects.
In addition since $f$ is proper, by remark 7.3.1.5 in HTT%
~\cite{doi:10.1515/9781400830558}, $f_\ast$ is $\omega$\=/continuous so that
$f^\ast$ preserves $\omega$\=/compact objects. We can then apply
\cref{LimDiagExp}.
\end{proof}

We now describe subtoposes of an exponentiable $\infty$\=/topos.

\begin{proposition}
Let $\mathcal X$ be an exponentiable $\infty$\=/topos and $i
: \mathcal Y \subset \mathcal X$ be a subtopos. If the reflective localisation
$i^\ast : \mathcal S\mathrm{h}(\mathcal X) \to \mathcal S\mathrm{h}(\mathcal
Y)$ is $\omega$\=/accessible, then $\mathcal Y$ is exponentiable.
\end{proposition}

\begin{proof}
If the right adjoint to $i^\ast$ is $\omega$-continuous, then
$\mathcal S\mathrm{h}(\mathcal Y)$ becomes a retract by $\omega$\=/continuous
functors and we conclude with \cref{RetractOfCont}.
\end{proof}

\begin{corollary}
Let $\mathcal X \hookrightarrow \mathcal Y$ be a closed
subtopos of $\mathcal Y$. Suppose $\mathcal Y$ is exponentiable, then $\mathcal
X$ is also exponentiable.
\end{corollary}

\begin{proof}
By remark 7.3.1.5 in HTT%
~\cite{doi:10.1515/9781400830558}, if $f$ is proper,
the functor $f_\ast$ is $\omega$-continuous.
\end{proof}

Finally combining the results we have on open and closed subtoposes, we get the
following proposition:

\begin{proposition}
Every locally closed subtopos of an exponentiable
$\infty$\=/topos is exponentiable.
\end{proposition}

\begin{proposition}
Let $f : \mathcal X \to \mathcal Y$ be a map between two
$\infty$\=/toposes. Suppose moreover that $f$ is cell-like or that $f$ is
essential with $\mathcal X$ having trivial $\mathcal Y$-shape. In such
circumstances, if $\mathcal X$ is exponentiable, then $\mathcal Y$ is also
exponentiable.
\end{proposition}

\begin{proof}
In both cases, $f^\ast$ is fully faithful with a (left or right)
adjoint that commutes with filtered colimits. Then apply
\cref{RetractOfCont}.
\end{proof}

\begin{remark}
By a result of \textsc{Scott}%
~\cite{doi:10.1007/BFb0073967}, every exponentiable locale has enough points.
This is no longer the case for $\infty$\=/toposes%
~\cite[Example 6.5.4.5]{doi:10.1515/9781400830558}.
\end{remark}

\section{Dualisability of the
\texorpdfstring{$\infty$\=/category}{infinity-category} of stable sheaves}

In this section we prove that when an $\infty$\=/topos is exponentiable, its
$\infty$\=/category of stable sheaves is dualisable.

\subsection{Stabilisation for presentable
\texorpdfstring{$\infty$\=/categories}{infinity-categories}}

We shall recall the definition of the
stabilisation functor and its properties. Our reference for this
topic is \emph{Higher Algebra}%
~\cite[Ch. 1 \& Sec.~4.8]{lurie2017higher}. We shall denote by $\mathrm{Sp}$
the $\infty$\=/category of spectra.

\begin{definition}
Let $\mathcal P\mathrm{res}_\mathrm{st}$ denote the full subcategory of
$\widehat{\mathcal C\mathrm{at}}_\mathrm{cc}$ whose objects are the
presentable and stable
$\infty$\=/categories.
\end{definition}

\begin{theorem}[{\cite[4.8.1.23 \& 4.8.2.18]{lurie2017higher}}]
The $\infty$\=/category
$\mathcal P\mathrm{res}_\mathrm{st}$ inherits a closed symmetric monoidal
structure from $\mathcal P\mathrm{res}$.
Furthermore, the inclusion functor $
\mathcal P\mathrm{res}_\mathrm{st} \hookrightarrow \mathcal P\mathrm{res}$
has a left adjoint, the
stabilisation functor:
$\mathcal C \mapsto \mathrm{Sp}(\mathcal C) \simeq
\mathcal C \otimes \mathrm{Sp}$
making $\mathcal P\mathrm{res}_{\mathrm{st}}$ a symmetric monoidal reflective
localisation of $\mathcal P\mathrm{res}$.
\end{theorem}

\subsection{Dualisability in \texorpdfstring{$\mathcal P\mathrm{res}$}{Pr}}

We start by recalling the notion of dualisable objects in a symmetric monoidal
$\infty$\=/category $(\mathcal C, \otimes)$%
~\cite[Ch. 4.6.1]{lurie2017higher}.

\begin{definition}
An object $X$ of $\mathcal C$ is dualisable if there exists
another object $X^\vee \in C$ with two maps
$\eta : 1_{\mathcal C} \to X \otimes X^\vee$
and $\varepsilon : X^\vee \otimes X \to 1_{\mathcal C}$.
where $1_{\mathcal C}$ is the unit of $\mathcal C$, such that the composite
maps:
\[
	\begin{tikzcd}
		X \arrow[rr, " \eta\thinspace \otimes\thinspace \mathrm{Id}_X"]
		&& X \otimes X^\vee \otimes X
		\arrow[rr, "\mathrm{Id}_X \thinspace \otimes\thinspace \varepsilon"]
		&& X\thinspace;
		\\
		X^\vee
		\arrow[rr, "\mathrm{Id}_{X^\vee}\thinspace \otimes\thinspace \eta"] 
		&& X^\vee \otimes X \otimes X^\vee
		\arrow[rr, "\varepsilon\thinspace \otimes\thinspace \mathrm{Id}_{X^\vee}"]
		&& X^\vee\thinspace,
	\end{tikzcd}
\]
are homotopic to the identities on $X$ and $X^\vee$ respectively.
\end{definition}

\begin{remark}
In the case where $\mathcal C$ is a closed symmetric
monoidal $\infty$\=/category, a dualisable object $X$ has its dual given
by $X^\vee = [X, 1]$ where $[-,-]$ is the internal hom associated
to the monoidal structure and $1$ is the monoidal unit.
\end{remark}

\begin{lemma}%
\label{Rectractdualisable}
In a closed symmetric monoidal $\infty$\=/category,
any retract of a dualisable object is dualisable.
\end{lemma}

\begin{proof}
Let $r : X \to Y$ be a retraction with $X$ a dualisable object
and let $s : Y \to X$ be a section. Set $Y^\vee = [Y, 1_{\mathcal C}]$ an let's
show that $Y^\vee$ has the right property. Because $ r : X \to Y$ is a
retraction, the same is true for $s^\vee : X^\vee \to Y^\vee$.
We are then supplied with maps $\eta_Y = (r \otimes s^\vee)\eta_X$ and
$\varepsilon_Y = \varepsilon_X(r^\vee \otimes s)$.
The composition $(\mathrm{Id}_Y \otimes \varepsilon_Y)\circ (\eta_Y \otimes
\mathrm{Id}_X): Y \to Y$ is then a retract of $\mathrm{Id}_X$, hence homotopic
to the identity itself.  The same is true for the other composition.
\end{proof}

\begin{theorem}%
\label{thm:dualisabilite_dans_les_categories_presentables}
The $\infty$\=/categories of the form
$\mathcal P(D)$ with $D$ a small $\infty$\=/category and
their retracts are dualisable objects of
$\widehat{\mathcal C\mathrm{at}}_\mathrm{cc}$. Moreover, they are exactly
the dualisable objects of $\mathcal P\mathrm{res}$.
\end{theorem}

\begin{proof}
Let $D$ be a small $\infty$\=/category, then if $\mathcal P(D)$ has
a dual, it has to be $\mathcal P(D^\mathrm{op})$, so let's introduce $\mathcal
P(D^\mathrm{op} \times D)$ the $\infty$\=/category of bimodules on $D$; we have
$\mathcal P(D^\mathrm{op}) \otimes \mathcal P(D)
\simeq \mathcal P(D^\mathrm{op} \times D)$.  

Then let $\eta : \mathcal S \to \mathcal P(D^\mathrm{op} \times D)$ be the
cocontinuous functor
sending the point $\ast \in \mathcal S$ to the
map-bimodule ${[-,-]}_D$.  And finally, let
$\varepsilon : \mathcal P(D^\mathrm{op} \times D) \to \mathcal S$
be the coend functor.

The composition $(\mathrm{Id}\otimes \varepsilon)(\eta \otimes \mathrm{Id})$
is given by the formula for presheaves:
$\int_{b\thinspace \in\thinspace D^\mathrm{op}} F(b) \times [a,b] = F(a)$
for a functor $F \in \mathcal P(D)$. 
The sister formula comes from
$\int_{a\thinspace \in\thinspace D}  [a,b] \times F(a) = F(b)$
for a functor $F \in \mathcal P(D^\mathrm{op})$.
Because a retract of a dualisable object is dualisable by
\cref{Rectractdualisable}, we are done for the first half.

Let $\mathcal C$ be a dualisable presentable $\infty$\=/category,
$D \subset \mathcal C$ be a small and dense subcategory and let $L :
\mathcal P(D) \to \mathcal C$ be the associated reflective localisation functor.
The dual map $L^\vee : \mathcal C^\vee
\simeq {[\mathcal C , \mathcal S]}_\mathrm{cc}
\to \mathcal P(D^\mathrm{op})$ is fully faithful because $L$ is a reflective
localisation functor, it is also cocontinuous. It has a left adjoint which is
the left Kan extension along $L$. As a consequence $\mathcal C^\vee$ is a
retract by cocontinuous functors of $\mathcal P(D^\mathrm{op})$.
Finally because $\mathcal C \simeq {(\mathcal C^\vee)}^\vee$ we deduce that
$\mathcal C$ is a retract of $\mathcal P(D)$.
\end{proof}

\subsection{Dualisability of stable sheaves}

The $\infty$-logos of an exponentiable $\infty$\=/topos is not
dualisable in general in $\mathcal P\mathrm{res}$,
as in general
an $\infty$\=/category of ind-objects is not dualisable. This is no longer
the case in $\mathcal P\mathrm{res}_\mathrm{st}$.

\begin{theorem}%
\label{thm:objets_dualisables_dans_pr_st}
The dualisable objects of $\mathcal P\mathrm{res}_\mathrm{st}$ are the
$\infty$\=/categories of the form $\mathcal P(D) \otimes \mathrm{Sp}$ and
their retracts.
\end{theorem}

\begin{proof}
Since the stabilisation functor
$(\mathrm{Sp} \otimes -) : \mathcal P\mathrm{res}
\to \mathcal P\mathrm{res}_\mathrm{st}$
is symmetric monoidal, it sends
dualisable objects to dualisable objects. Hence we know that the
$\infty$\=/categories of the form $\mathcal P(D) \otimes \mathrm{Sp}$ and
their retracts are dualisable by
\cref{thm:dualisabilite_dans_les_categories_presentables}.

Let $\mathcal C$ be a dualisable presentable stable $\infty$\=/category,
then there exists a small $\infty$\=/category $D$ and a reflective localisation
$\mathcal P(D) \to \mathcal C$ which induces a reflective localisation
$\mathcal P(D) \otimes \mathrm{Sp} \to \mathcal C \otimes \mathrm{Sp}
\simeq \mathcal C$. We end the proof with the same arguments as in
\cref{thm:dualisabilite_dans_les_categories_presentables}.
\end{proof}

\begin{lemma}%
\label{thm:ind_st_est_retract_de_prefaiceaux}
Let $D$ be a small $\infty$\=/category
with finite colimits. Then $\mathrm{Ind}(D) \otimes \mathrm{Sp}$ is a
retract of $\mathcal P(D) \otimes \mathrm{Sp}$ in
$\mathcal P\mathrm{res}_\mathrm{st}$.
\end{lemma}

\begin{proof}
Since $D$ has small colimits, then by
\cref{TensorProductFormula} the $\infty$\=/category $\mathrm{Ind}(D) \otimes
\mathrm{Sp}$ is equivalent to the $\infty$\=/category of left exact functors
${[D^\mathrm{op}, \mathrm{Sp}]}^\mathrm{lex}$ and
$\mathcal P(D) \otimes \mathrm{Sp}$ is equivalent to
$[D^\mathrm{op}, \mathrm{Sp}]$.

Because $\mathrm{Sp}$ is stable and colimits in functor $\infty$\=/categories
are computed pointwise, the embedding 
${[D^\mathrm{op}, \mathrm{Sp}]}^\mathrm{lex}
\hookrightarrow [D^\mathrm{op}, \mathrm{Sp}]$
commutes with all limits and colimits.
It then has a left adjoint such that $\mathrm{Ind}(D) \otimes \mathrm{Sp}$ is a
retract of $\mathcal P(D) \otimes \mathrm{Sp}$ by cocontinuous functors.
\end{proof}

\begin{theorem}%
\label{DualisabilityOfStableSheaves}
Let $\mathcal X$ be an exponentiable $\infty$\=/topos, then
$\mathcal S\mathrm{h}(\mathcal X) \otimes
\mathrm{Sp}$, the $\infty$\=/category of stable sheaves on $\mathcal X$,  is a
dualisable object of $\mathcal P \mathrm{res}_{\mathrm{St}}$.
\end{theorem}

\begin{proof}
After tensoring by $\mathrm{Sp}$ a standard presentation of
$\mathcal S\mathrm{h}(\mathcal X)$,
we get a retraction in $\mathcal P\mathrm{res}_{\mathrm{St}}$:
\[
	\begin{tikzcd}
		\mathrm{Ind}(D) \otimes \mathrm{Sp}
		\arrow[r, "{\varepsilon'}", shift right, swap] 
		& \mathcal S\mathrm{h}(\mathcal X) \otimes \mathrm{Sp}\ .
		\arrow[l, "{\beta'}", shift right, swap]
	\end{tikzcd}
\]
We conclude using \cref{thm:ind_st_est_retract_de_prefaiceaux} and
\cref{thm:objets_dualisables_dans_pr_st}.
\end{proof}

This theorem can be compared to a result of \textsc{Niefield} and
\textsc{Wood}%
~\cite{niefield2017coexponentiability}:

\begin{theorem*}
An $R$-ring $A$ is coexponentiable if and only if $A$ is projective and
finitely generated as an $R$-module.
\end{theorem*}


\addtocontents{toc}{\protect\vspace{\beforebibskip}} 
\addcontentsline{toc}{section}{\refname}             
\bibliographystyle{bib}
\bibliography{bib.bbl}

\end{document}